\documentclass[a4paper]{amsart}
\usepackage[textwidth=15cm,top=3cm, bottom=3cm, hcentering]{geometry}
\usepackage[colorlinks=true,linkcolor=red,citecolor=blue]{hyperref}

\usepackage{amsmath,amscd,amsthm,amssymb}
\usepackage{color}
\usepackage[utf8]{inputenc}
\normalfont
\usepackage[T1]{fontenc}
\usepackage{enumerate}

\usepackage{tikz,tikz-cd}
\usetikzlibrary{matrix,calc,positioning,arrows,decorations.pathreplacing,patterns,arrows}

 \usepackage[all]{xy}

\newtheorem{theorem}{Theorem}[section]
\newtheorem{lemma}[theorem]{Lemma}
\newtheorem*{theorem*}{Theorem}
\newtheorem{proposition}[theorem]{Proposition}
\newtheorem{corollary}[theorem]{Corollary}
\newtheorem*{corollary*}{Corollary}

\newtheorem{thmintro}{Theorem}

\theoremstyle{definition}
\newtheorem{definition}[theorem]{Definition}

\theoremstyle{remark}
\newtheorem{example}[theorem]{Example}

\newtheorem{remark}[theorem]{Remark}

\newcommand{\CC}{\mathbb{C}}

\newcommand{\FF}{\mathbb{F}}

\newcommand{\KK}{\mathbb{K}}

\newcommand{\NN}{\mathbb{N}}

\newcommand{\PP}{\mathbb{P}}
\newcommand{\QQ}{\mathbb{Q}}
\newcommand{\RR}{\mathbb{R}}

\newcommand{\ZZ}{\mathbb{Z}}
\newcommand{\Aa}{\mathcal{A}}

\newcommand{\Ff}{\mathcal{F}}

\newcommand{\Ll}{\mathcal{L}}
\newcommand{\Mm}{\mathcal{M}}

\newcommand{\Oo}{\mathcal{O}}

\newcommand{\Uu}{\mathcal{U}}

\newcommand{\Ww}{\mathcal{W}}
\newcommand{\Xx}{\mathcal{X}}

\newcommand{\Zz}{\mathcal{Z}}

\newcommand{\st}{^{s\textrm{-tame}}}
\newcommand{\spt}{^{(s+1)\textrm{-tame}}}
\newcommand{\smt}{^{(s-1)\textrm{-tame}}}

\newcommand{\Sym}{\operatorname{Sym}}

\newcommand{\cat}[1]{{\mathsf{#1}}}

\renewcommand{\rm}[1]{{\mathrm{#1}}}
\newcommand{\lra}{\longrightarrow}

\newcommand{\op}{\rm{op}}

\newcommand{\Map}{\rm{Map}}

\newcommand{\Hom}{\rm{Hom}}

\newcommand{\id}{\rm{id}}
\newcommand{\kk}{\mathbf{k}}

\newcommand{\Ker}{\mathrm{Ker}}
\newcommand{\Dec}{\mathrm{Dec}}

\newcommand{\TW}{\text{\Tiny{$\mathrm{TW}$}}}

\title[{O\lowercase{n the $\ell$-adic homotopy type of configuration spaces}}]{{\Large O\lowercase{n the $\ell$-adic homotopy type of configuration spaces}}}

\author{Joana Cirici}
\author{Geoffroy Horel}

\address[J. Cirici]{Departament de Matemàtiques i Informàtica, Universitat de Barcelona and
Centre de Recerca Matemàtica (Barcelona, Spain)
}
\email{jcirici@ub.edu}

\address[G. Horel]{Université Sorbonne Paris Nord, Laboratoire Analyse, Géométrie et Applications, CNRS (UMR 7539), and Institut Universitaire de France}
\email{horel@math.univ-paris13.fr}

\thanks{
J. Cirici acknowledges financial support from
Govern de Catalunya (Serra H\'{u}nter Program) and the Spanish State Research Agency (CEX2020-001084-M and PID2024-155646NB-I00).
G. Horel acknowledges financial support from Agence Nationale pour la Recherche through project ANR-25-CE40-28616 Kash.
}

\begin{document}
	
\begin{abstract}We give algebraic models for the tame homotopy type of the configuration spaces
 of certain algebraic varieties of Tate type. Such tame models carry information on the $\ell$-adic homotopy type. Our method uses the theory of weights in étale cohomology, and also produces models for more general arrangement complements, both in the tame sense and over the rationals.
\end{abstract}
	
\maketitle
\tableofcontents
\section{Introduction}
Understanding the homotopy type of configuration spaces of manifolds is a long-standing problem in algebraic topology. The main conjecture in this area is that the homotopy type of $\mathrm{Conf}_n(X)$, the configuration space of $n$ points in a manifold $X$, only depends on the homotopy type of the manifold, if $X$ is closed and simply-connected. There is some evidence for this conjecture. It holds for the stable homotopy type of the configuration spaces as well as for the homotopy type of the loop space of the configuration spaces by work of Aouina--Klein and Levitt respectively (see \cite{aouinahomotopy} and \cite{levittspaces}). It follows that both homology  and homotopy groups of configuration spaces are homotopy invariants in the case of simply-connected closed manifolds.

If we restrict to the rational homotopy type, there are stronger known results. First, 
for complex algebraic varieties, and building on the works of Fulton--MacPherson \cite{FMP} and Totaro \cite{Totaro}, Kriz \cite{Kriz} gave an explicit rational model for their configuration spaces, given by the penultimate page of the Leray spectral sequence for the
constant sheaf $\underline{\QQ}$ relative to the inclusion
 $\mathrm{Conf}_n(X)\hookrightarrow X^n$.
This model also coincides with the Lambrechts--Stanley model (see \cite{lambrechtsremarkable}) which is conjectured to be a model for all simply-connected closed oriented manifolds. Using a generalization of Kontsevich integration, Idrissi  \cite{idrissilambrechts} showed that this conjecture is true when extending scalars to $\RR$. In particular, the real homotopy type of the configuration space depends only on the real homotopy type of the manifold (see also  \cite{CaWill}).

For integral homotopy types, such results seem out of reach, starting from the fact that there are no reasonable algebraic models in the purely integral setting.
Tame homotopy theory, initiated by Dwyer \cite{Dwyer}  and Anick \cite{Anick} among others,
provides an intermediate framework between integral and rational homotopy theory.
It captures the part of the integral homotopy type of a topological space that is insensitive to Steenrod operations while allowing for a manageable theory of algebraic models.
In this work, we study tame homotopy types for configuration spaces of algebraic varieties, using the theory of weights in étale cohomology. The results provide new evidence for the validity of the configuration space conjecture, for integral homotopy types.

\medskip

Classically, tame homotopy theory is done with $\ZZ$ coefficients. There is a Sullivan-type theory
based on a filtered commutative dg-algebra constructed by Cenkl and Porter \cite{CePo} and further studed by Scheerer in \cite{Sch1}. These works are not sufficiently flexible for our purposes as we will work with étale cochains over $\ZZ_\ell$. We propose an alternative approach to tame homotopy theory using $E_\infty$-algebras and prove a Mandell-type Theorem in this setting.

Let us briefly outline this result more precisely. We complete at the prime $\ell$ and
 also fix integers $r$ and $s$,
accounting for the connectivity and a tame parameter respectively.
We introduce
$(r,s)$-tame homotopy theory as the study of $(r-1)$-connected $\ell$-complete spaces up to the notion of \textit{$(r,s)$-tame weak equivalence} defined as follows: we say that a map
$f:X\to Y$ between based $(r-1)$-connected $\ell$-complete spaces
is an \textit{$(r,s)$-tame weak equivalence} if the map $\pi_{r+k}(f)$ is an isomorphism for $k< s$ and if $\pi_{r+k}(f)\otimes\QQ$ is an isomorphism for all values of $k$. When $s=0$ the theory reduces to rational homotopy, and increasing $s$ restores more of the integral homotopy type of the space.
In order to provide models for tame homotopy types we introduce a notion of \textit{$(r,s)$-tame quasi-isomorphism} for filtered $E_\infty$-algebras, asking for isomorphisms up to weight $r+s$ as well as a rational isomorphism in every weight (see Definition
\ref{defi : s-tame weak equivalence}).
We prove:

\begin{thmintro}[Theorem \ref{theo : tame mandell}]
Let $r\geq 3$ and $s\leq 2\ell-3$ or $r= 2$ and $s\leq \ell-2$. Let $X$ and $Y$ be two $(r-1)$-connected $\ell$-complete based spaces with finite type cohomology. If there exists an $(r,s+\lfloor\frac{s+1}{r-1}\rfloor)$-tame quasi-isomorphism of $E_\infty$-algebras
\[(\tilde{C}^*(X,\ZZ_\ell),\tau)\simeq (\tilde{C}^*(Y,\ZZ_\ell),\tau)\]
then $X$ and $Y$ have the same $(r,s)$-tame homotopy type.
\end{thmintro}
Here, $\tau$ denotes the canonical filtration. Note that the tame parameter in the algebraic side is higher than in the space side, so we are paying an exchange rate which depends on the connectivity of the spaces. The bounds on $s$ above are imposed by the fact that $2\ell-3$ is the lowest degree Steenrod operation at the prime $\ell$.
It is also the lowest dimension in which the homotopy groups of spheres have $\ell$-torsion.
One novelty of this result in contrast with the existing literature on tame homotopy theory is the validity of the result for $r=2$. In fact, with our approach, in the non-simply-connected case, we can also recover tame information on the fundamental group (see Section \ref{subsec: fundgp}).

Let $\Xx$ be a smooth variety defined over a $p$-adic field $K$
 (a finite extension of $\QQ_p$) with residue field $\FF_q$, so that 
 $q=p^k$.
 Assume that a choice of an embedding of $K$ in $\CC$ has been made and let $X:=\Xx\otimes_{K}\CC$ denote the base-change of $\Xx$. As a complex algebraic variety, $X$ has an underlying complex analytic space which we denote by $X$ as well. 
 For a fixed prime $\ell\neq p$,
 the $E_\infty$-algebra of étale cochains 
  $C^*_{et}(\Xx;\ZZ_\ell)$ is naturally quasi-isomorphic to the singular cochains $C^*(X,\ZZ_\ell)$ of $X$.
  It has the advantage of being equipped
 with a Frobenius endomorphism $\varphi$ arising from the absolute Galois group $\mathrm{Gal}(\overline{K}/K)$. This will be used to prove splitting results.

For our main results below, we restrict to those varieties for which
the eigenvalues of the Frobenius endomorphism acting on étale cohomology satisfy a purity property that depends on the tame parameters $r$ and $s$.
Namely,
we say that a cohomologically $(r-1)$-connected variety $\Xx$ is \textit{$(r,s)$-tamely pure} if
the eigenvalues of $\varphi$ acting on $H^i_{et}(\Xx;\FF_\ell)$ are in $\{(\pm\sqrt{q})^i\}$ for $i< s+r$ and the eigenvalues of $\varphi$ acting on $H^i_{et}(\Xx;\QQ_\ell)$ are Weil numbers of weight $i$ for all $i$.
It follows from Deligne's proof of the Weil conjectures that smooth and proper varieties with good reduction are always $(r,0)$-tamely pure.
There are also many examples of varieties that are tamely pure for any $s\geq 0$, which we call of \textit{Tate type}. These include $\CC\PP^n$, $\overline{\Mm}_{0,n}$ and, more generally, blow-ups of Tate type varieties.

The tame homotopy type admits a natural notion of formality, and we extend the principle that ``purity implies formality'' (see \cite{DGMS}, \cite{GNPR}, \cite{Petersen}, \cite{CiHo1}, and \cite{CiHo2}) to the tame setting. As noted by Mandell, the $E_\infty$-algebra of cochains of a space in positive characteristic cannot be formal unless the space has contractible components.
Tame formality provides a meaningful weakening of $E_\infty$-formality, distinct from the theory of $E_n$-formality proposed by Mandell.

For the purpose of readibility, in this introduction we will not state explicitly the tame parameters.
We prove:
\begin{thmintro}[Theorem \ref{theo : tameprojective} and Corollary \ref{coro: pureformaltate}]
Assume that $\Xx$ is of Tate type.  If $X$ is simply-connected, then the tame homotopy type of $X$ is a formal consequence of the cohomology ring $H^*(X,\ZZ_\ell)$.
\end{thmintro}

The above theorem is also valid for a more general notion of purity
as in \cite{CiHo1}, which includes complements of hyperplane and toric arrangements, and the uncompactified moduli spaces $\Mm_{0,n}$. Also, in the non-simply-connected case, we can extract fundamental group information
from the cohomology (see Section \ref{subsection : fundamental group}).
To the best of our knowledge, this is the first investigation of tame formality in the literature. It extends previous formality results for associative algebras
with positive characteristic coefficients,
to the $E_\infty$-setting.

The proof of the above theorem relies on a splitting theorem for Weil algebras. In non-pure situations, this splitting theorem is also used to produce finite dimensional models for complements of arrangements.
We next explain the particular case of configuration spaces.

Denote by $i:\mathrm{Conf}_n(X)\hookrightarrow X^n$ the natural inclusion. The  Leray spectral sequence for the constant sheaf $\underline{\ZZ}_\ell$ relative to this inclusion has a single page  $({}^{\Ll}E_{2m}^{*,*}(\ZZ_\ell),d_{2m})$ with non-trivial differential, and this page
is a strictly commutative dg-algebra computing the cohomology ring of the configuration space. Here $m$ is the dimension of $X$.
We prove:

\begin{thmintro}[Theorem \ref{theo: mainconftame} and Corollary \ref{coro: mainconftate}]
Assume that $\Xx$ is of Tate type.  Then the tame homotopy type of $\mathrm{Conf}_n(X)$ is determined by the commutative dg-algebra
$({}^{\Ll}E_{2m}^{*,*}(\ZZ_\ell),d_{2m})$ up to tame quasi-isomorphism.
\end{thmintro}

The precise values of the tame parameters for $(r,s)$-tamely pure varieties are given in Theorem \ref{theo: mainconftame}.
In the case $s=0$ this recovers the rational model of Totaro--Kriz for configuration spaces of smooth complex projective varieties. We actually prove, in Theorem \ref{maintame}, a more general statement for arrangement complements.
This in particular gives a tame version of Morgan's model for smooth quasi-projective varieties.

Note that our proof is functorial. In particular, for an inclusion of smooth tamely pure varieties it gives a model for the induced map of configuration spaces.
For $\ell$-complete homotopy types, the above theorem gives:

\begin{corollary*}[Corollary \ref{coro: tatetameconfeladic}]
Let $\Xx$ be a simply-connected smooth variety of Tate type over a number field $K$, and of dimension $m\geq 2$.
Then for a large enough $\ell$, the $\ell$-complete homotopy type of $\mathrm{Conf}_n(X)$ depends only on the cohomology ring $H^*(X;\ZZ_\ell)$ together with the diagonal class $\Delta\in H^{2m}(X\times X;\ZZ_\ell)$.
\end{corollary*}

\subsection*{Perspectives}
Our purity implies formality statements should imply tame formality of the little disks operad, improving the partial formality results proven in \cite{BoHo1}.
Likewise, one should obtain tame formality of the framed version,
following the techniques of \cite{BoCiHo}.
These formality results and the implications in manifold calculus will be explored by Nicolas Guès in his forthcoming PhD thesis.
More generally, one could expect that our tame models for configuration spaces can be assembled into a tame model for the configuration category \cite{BdBW}, giving applications to manifold calculus.

\subsection*{Organization of the paper}
In Section \ref{sec: rationale} we give an overview of the main results of the paper, purely in the context of rational homotopy theory. The main theorem is Theorem \ref{maintheorational} which gives a model for complements of arrangements of algebraic varieties via a certain page of the Leray spectral sequence that depends on the codimension of the arrangements. This specializes to configuration spaces, recovering the results of Totaro--Kriz and to smooth varieties, recovering the results of Morgan.
In Section \ref{sectiontame} we develop tame homotopy theory, proving the Mandell-type Theorem \ref{theo : tame mandell}. We also study the notion of tame formality and the relation with truncated algebras.
These first two sections may be read independently.
Lastly, Section \ref{tamevars} combines the material of the previous sections to prove the main results of this paper: Theorem \ref{theo : tameprojective} on purity implies formality, Theorem \ref{maintame} for tame models of arrangement complements and Theorem \ref{theo: mainconftame} on configuration spaces.

\section{Rational models of algebraic varieties}\label{sec: rationale}
In this section we use sheaf theory and the theory of weights to describe the rational homotopy type of certain arrangement complements of algebraic varieties. The description we obtain specializes to both smooth varieties, giving an $\ell$-adic proof of Morgan's result \cite{Morgan}, and to configuration spaces, recovering the Totaro--Kriz  model \cite{Kriz} purely using weight arguments.
Additionally, it serves as a prelude to the (more involved) tame models developed later.

\subsection{Rationale}

Let $X$ be a smooth complex algebraic variety. By Hironaka's resolution of singularities one may find a good compactification $g:X\hookrightarrow \overline{X}$, so that the complement $D:=\overline{X}-X$ is a simple normal crossings divisor. 
As shown by Deligne \cite{DeHII}, 
the Leray spectral sequence for the constant sheaf $\underline{\QQ}_X$ relative to this inclusion 
\[^{\Ll}E_2^{i,j}(g,\QQ):=H^i(\overline{X},R^jg_*\underline{\QQ}_X)\Longrightarrow H^{i+j}(X;\QQ)\]
always degenerates at the third page. Moreover, the second page admits a simple description in terms of the cohomology
\[^{\Ll}E_2^{i,j}(g,\QQ)\cong H^{i}(D^{(j)},\QQ)\]
of the canonical simplicial hyperresolution $D^{(\bullet)}\to D$.
The differential $d_2$ is given by the alternating sum of Gysin maps 
associated to the distinct inclusions in the simplicial hyperresolution of $D$
and the algebra structure is induced by the alternating sum of restriction morphisms together with the ring structure of $H^*(D^{(\bullet)};\QQ)$.
Morgan showed in \cite{Morgan} that the cdga $({}^{\Ll}E_2^{*,*}(g,\QQ),d_2)$ is a model for the $\QQ$-homotopy type of $X$.
This has several important consequences. First, as it is given by cohomologies of smooth projective varieties, it gives in particular a finite dimensional model. Moreover, it shows that rational models for smooth varieties carry positive weights and so in particular one obtains restrictions on homotopy types of smooth varieties. For instance, the real Malcev Lie algebra of the fundamental group is determined by its quotient by the
fifth order commutators.

In \cite{FMP}, Fulton and MacPherson 
gave an explicit description for a good compactification of
\[\mathrm{Conf}_n(X):=\{(x_1,\cdots,x_n)\in X^n; x_i\neq x_j\},\]
the unordered configuration space of $n$ points in a smooth complex projective variety $X$.
Together with Morgan's result, this gives a model for the $\QQ$-homotopy type of $\mathrm{Conf}_n(X)$ which depends only on the homotopy type of $X$.

There is another (non-good) compactification $f:\mathrm{Conf}_n(X)\hookrightarrow X^n$ given by the obvious inclusion.
Totaro \cite{Totaro} studied the Leray spectral sequence  
for the constant sheaf relative to this inclusion and showed that the only non-trivial differential is $d_{2m}$, where $m$ is the dimension of $X$. Moreover, Totaro gave a description of the cdga 
$({}^{\Ll}E_{2m}^{*,*}(f,\QQ),d_{2m})$ in terms of the cohomology of $X^n$.
By work of Kriz \cite{Kriz}, this algebra is a $\QQ$-model for the configuration space, which is actually simpler than the model of Fulton and MacPherson. 
The proof of Kriz follows by building an explicit quasi-isomorphism 
from this algebra to the model obtained by Fulton and MacPherson. 

In this paper, we reprove the result of Kriz by generalizing Morgan's theory to include more general compactifications.
This gives a more conceptual proof, based on the theory of weights, which allows for an adaptation to the study of tame homotopy types. 
More precisely, we study the Leray spectral sequence for the obvious inclusion of complements of codimension-$c$ arrangements in a smooth projective variety, where $c\geq 1$. These are defined
inductively as follows:

\begin{definition}\label{defarrangement}Let $X$ be a smooth proper variety.
An arrangement of distinct subvarieties $Z=\{Z_i\}_{1\leq i\leq s}$ in $X$ is said to be \textit{admissible of codimension $c$ in $X$}
if 
\begin{enumerate}
 \item 
each $Z_i$ is smooth of codimension $c$ and 
\item for any finite set of indices $I\subseteq \{1,\cdots,s\}$ and $j\notin I$,  
 the set-theoretic
intersection
\[\left(\bigcup_{i\in I}Z_i\right)\cap Z_j\]
is empty or it is the union of an admissible codimension $c$ arrangement
in each $Z_j$.
\end{enumerate}
\end{definition}

Two particular situations of interest will be the following:
\begin{itemize}
 \item (Good compactifications). Let $g:X\to \overline{X}$ be a good compactification of a smooth variety $X$.
 Then $\overline{X}-X$ is an admissible arrangement of codimension 1 in $\overline{X}$.
 \item (Configuration spaces). Given a smooth complex projective variety $X$, let 
 \[Z_{ij}:=\{(x_1,\cdots,x_n)\in X^n; x_i=x_j\}.\] Then $Z=\{Z_{ij}\}$ is an admissible arrangement of codimension $m$ in $X^n$, where $m$ is the complex dimension of $X$, and 
 $U:=X^n-Z=\mathrm{Conf}_n(X)$.
\end{itemize}

Given an admissible arrangement $\{Z_i\}$ in $X$, denote by $U=X- \bigcup_i Z_i$ the complement and by 
$f:U\hookrightarrow X$ the natural inclusion.
As shown by Weber \cite{Weber}, the Leray spectral sequence 
\[^{\Ll}E_2^{i,j}(f,\QQ)=H^i(\overline{X},R^jf_*\underline{\QQ}_X)\Longrightarrow H^{i+j}(X;\QQ).\]
 for the constant sheaf $\underline{\QQ}_X$ relative to this inclusion 
 satisfies $d_i=0$ for all $i\neq 2c$. In particular, the cdga $(^{\Ll}E_{2c}^{*,*}(f,\QQ),d_{2c})$ computes the cohomology  of $U$ with $\QQ$-coefficients.

Weber's degeneration result is a generalization of both degeneration results due to Deligne and Totaro respectively. In the three cases, the main ingredient in the proof is the theory of weights, which is also essential in Morgan's rational models, through mixed Hodge theory. 
We will not be more original
here. Our strategy is to use Galois actions in étale cohomology to show that a certain sheaf-theoretic model for $U$, together with a certain canonically defined filtration, is homotopically split and, as a consequence, quasi-isomorphic to 
the commutative dg-algebra given by the penultimate page of the Leray spectral sequence. For that, we first review some preliminaries on filtrations and develop the abstract machinery of weights on algebras, through a theory of Weil algebras. Finally, we apply this to a sheaf-theoretic model for complements of arrangements.

\subsection{Preliminaries on filtrations}Throughout this paper we will consider filtered complexes $(A,W)$ with $W$ an increasing non-negative filtration.
We will be using the following family of canonical filtrations:

\begin{definition}\label{Ciricifiltration}
For any integer $k\geq 1$, denote by $\tau^{(k)}$ the \textit{canonical filtration of speed $k$}, defined on a cochain complex $A$ by letting
 \[
\tau^{(k)}_iA^n=\left\{ 
\begin{array}{ll}
0&,ki<n\\
\Ker(d:A^n\to A^{n+1})&,ki=n\\
A^n&,ki>n
\end{array}
\right..
\]
\end{definition}

Note that $\tau^{(1)}=\tau$ is the usual canonical filtration. 
For all $k\geq 1$, the filtration $\tau^{(k)}$ defines a symmetric monoidal functor from complexes to filtered complexes. The following is straightforward:
\begin{lemma}\label{quisspeed}Let $A$ be a cochain complex and $k\geq 1$. For all $i\in\ZZ$, the associated graded complex $Gr_i^{\tau^{(k)}}A$ of weight $i$ is given by
\[A^{ki-(k-1)}/\ker d\to A^{ki-(k-2)} \to\cdots\to A^{ki-1}\to\ker d\cap A^{ki}.\]
In particular,
\[Gr_i^{\tau^{(k)}}
A\simeq \bigoplus_{j=0}^{k-1}H^{ki-j}(A)[-ki+j].\]
\end{lemma}

We will also use the décalage filtration, which generalizes the passage from the trivial to the canonical filtration:

\begin{definition}\label{defidec}The \textit{décalage} of a filtered complex $(A,W)$ is the filtered complex $(A,\Dec W)$ defined by
 \[\Dec W_i A^n:=W_{i-n}A^n\cap d^{-1}(W_{i-n-1}A^{n+1}).\]
\end{definition}
Note that décalage  defines a symmetric monoidal endofunctor $\Dec$ on the category of filtered complexes.

The spectral sequence $\{E_r(A,W),d_r\}$ associated to a filtered complex $(A,W)$ is given by
\[E_r^{-i,j}(A,W):={W_iA^{j-i}\cap d^{-1}(W_{i-r}A^{j-i+1})\over
 W_{i-1}A^{j-i}\cap d^{-1}(W_{i-r}A^{j-i+1})+d(W_{i+r-1}A^{j-i-1})\cap W_iA^{j-i}}.\]
and the differential $d_r:E_r^{-i,j}(A,W)\to E_r^{-i+r,j-r+1}(A,W)$ is induced by the differential $d$ of $A$. We have
\[E_0^{-i,j}(A,W)=Gr_i^WA^{j-i}\] and $E_\infty^{-i,j}(A,W)$ is naturally identified with $Gr_i^WH^{j-i}(A)$.

 The spectral sequences of a filtered complex and its décalage are related by a shift on their pages:
 
 \begin{lemma}[\cite{DeHII}]\label{dece0e1}
 There is a natural quasi-isomorphism of bigraded complexes
 \[E_0^{-i,j}(A,\Dec W)\lra E_1^{j-2i,i}(A,W)\]
inducing isomorphisms at the later pages.
 \end{lemma}

We now consider filtered commutative dg-algebras (cdga's for short) over a field $\kk$ of characteristic zero. 
For the following definition, we will
 view the associated graded functor
$(A,W)\mapsto Gr^W_*A=\bigoplus_j Gr^W_jA$
as an endofunctor in the category of filtered cdga's, with
the column filtration:
\[W_iGr^W_*A^n:=\bigoplus_{j\leq i}Gr_j^WA^n.\]

\begin{definition}
A filtered algebra $(A,W)$ is said to be \textit{homotopically split} if there is a string of filtered quasi-isomorphisms from 
$(A,W)$ to $(Gr^W_*A,W)$.
\end{definition}

Note in particular that if $(A,W)$ is homotopically split, then its associated spectral sequence degenerates at $E_{1}$. The converse is not true. The following example relates splittings and formality (see also \cite{CiriciGuillen} for the more general notion of $E_r$-formality).

\begin{example}
Consider the canonical filtration $\tau$
on an algebra $A$.
Then $(A,\tau)$ is homotopically split if and only if $A$  is a formal algebra.
\end{example}

\begin{example}
Morgan's result \cite{Morgan} on rational homotopy types may
be rephrased as the décalage of the weight filtration on $\Aa_{pl}(X)$ being homotopically split for a smooth variety $X$.
By \cite{CiriciGuillen} this is also true for singular varieties (see also \cite{CiHo1} for a general functorial statement).
\end{example}

\subsection{Weil algebras}We next introduce Weil algebras
as a multiplicative version of the $\ell$-adic counterpart of the notion of mixed
Hodge complex introduced by Deligne. The main property is that Weil algebras are always homotopically split.

Let $\ZZ_\ell$ be the ring of $\ell$-adic integers and fix $q=p^k$ which is a power of a prime number $p$ different from $\ell$.

\begin{definition}\label{defweilnum}
A \textit{Weil number of pure weight $i$} is an algebraic integer $x$ over $\ZZ_\ell$ such that for any embedding $\iota$ of $\overline{\QQ}_\ell$ in $\CC$, we have
\[|\iota(x)|=q^{i/2}.\]
\end{definition}

Observe that the product of two Weil numbers is a Weil number and the weight is additive with respect
to this operation. Note as well that the definition of Weil number depends on the fixed prime $\ell$ as well as the fixed prime power $q$, which we omit from the notation.

\begin{definition}
 A \textit{Weil algebra} is a filtered algebra $(A,W)$ of $\QQ_\ell[\varphi]$-modules such that:
 \begin{enumerate}[(i)]
  \item The filtration $W$ is increasing and non-negative.
  \item For all $i\geq 0$, $H^{*}(Gr_i^WA)$ is finite-dimensional.
  \item The only eigenvalues of the automorphism $\varphi^*$ induced on $H^{*}(Gr_i^WA)$ are Weil numbers of pure weight $i$.
 \end{enumerate}
\end{definition}

Note that, by weight reasons, the spectral sequence associated to any Weil algebra degenerates at the $E_1$-page. In fact, we have:

\begin{proposition}\label{Weilisformal}
 Weil algebras are functorially homotopically split.
\end{proposition}
\begin{proof}
 This will be proven later in a more general setting (Theorem \ref{e0formality} in the case $s=0$). Morally, this result boils down to elementary linear algebra: the splitting is simply given by writing the algebra as the direct sum of the generalized eigenspaces with respect to all the Weil numbers of a given weight.
\end{proof}

\begin{remark}
In fact, in Theorem \ref{e0formality}, the statement is infinity-categorical and is a priori weaker than what we are stating here. But it is in fact equivalent thanks to strictification results. We refer the reader to \cite[Remark 2.3]{BoCiHo} for details on this.
\end{remark}

\subsection{Sheaves in rational homotopy}
We next describe the 
$\kk$-homotopy type of any reasonable topological space, such as manifolds or analytic spaces, using sheaf theory, for $\kk$ a field of characteristic zero.

Let $f:X\to Y$ be a continuous map of topological spaces. Using the Thom-Whitney simple functor together with Godement's canonical sheaf resolution, Navarro-Aznar defined in \cite{Navarro} a derived direct image functor
$R_{\TW}f_*$ in the category of sheaves of cdga's over $\kk$. In the particular case of the map $\rho:X\to *$ to the point, this defines a derived global sections functor 
$R_{\TW}\Gamma(X,-)$ from sheaves of cdga's to cdga's, by letting $R_{\TW}\Gamma(X,-):=R_{\TW}\rho_*$. In order to lighten the notations, we will simply write $R$ instead of $R_{\TW}$ in what follows.

\begin{remark}\label{godemetale}
If instead of sheaves on a topological space, we consider sheaves in a topos with enough points, the Godement resolution gives a fibrant replacement in the category of sheaves of algebras (see \cite{Agusti}, \cite{Agusti2}). In particular, the derived direct image
$Rf_*$ and derived global sections $R\Gamma(X,-)$ are also defined multiplicatively in this case.
\end{remark}

Assume that $X$ is Hausdorff, paracompact and locally contractible and consider the 
 constant sheaf $\underline{\kk}_X$ on $X$. Then the cdga over $\kk$ given by
\[R\Gamma(X,\underline{\kk}_X)\]
is a model for the $\kk$-homotopy type of $X$ (see Theorem 5.5 of \cite{Navarro}). 
More generally:
\begin{proposition}\label{relativemodel}Let $X$ be a Hausdorff, paracompact and locally contractible topological space and
 $f:X\to Y$ a continuous map. Then the cdga 
\[ \Gamma(Y,Rf_*\underline{\kk}_X)\]
is  a model for the $\kk$-homotopy type of $X$.
\end{proposition}
This follows from the fact that, given a sheaf $\Ff$ of cdga's on $X$, there is a quasi-isomorphism of cdga's
\[R\Gamma(X,\Ff)\stackrel{\sim}{\lra}\Gamma(Y,Rf_*\Ff).\]

\subsection{Rational models for arrangement  complements}\label{filteredalgebrarational}

Let $Z=\{Z_i\}$ be an admissible codimension-$c$ arrangement in a smooth complex projective variety $X$, as in Definition \ref{defarrangement}.
Denote by $U=X- \bigcup_i Z_i$ the complement and by 
$f:U\hookrightarrow X$ the natural inclusion.

By Proposition \ref{relativemodel}, a model for the $\kk$-homotopy type of $U$ is given by 
\[\Aa_{\kk}(U):=\Gamma(X,Rf_*\underline{\kk}_U).\]
We will show that $\Aa_{\QQ_\ell}(U)$ is naturally quasi-isomorphic to 
the $E_{2c}$-page of the Leray spectral sequence for the constant sheaf over $\QQ_\ell$, relative to the inclusion $f$.
First, we make 
 the sheaf of algebras $Rf_*\underline{\QQ}_\ell$ on $X$ into a sheaf of filtered algebras 
 by considering the following filtration:
 
\[(\Aa_{\QQ_\ell}(U),W):=\Dec\, \Gamma(X,(Rf_*\underline{\QQ}_\ell,\tau^{(2c-1)})).
\]
Here $\tau^{(k)}$ denotes the speed-$k$ canonical filtration of Definition \ref{Ciricifiltration} and $\Dec$ is Deligne's décalage (see Definition \ref{defidec}).

Next, we lift this construction to the category of Weil algebras,
by reducing to $\QQ_\ell$-sheaves on the pro-étale site of $X$, as follows.
The arrangement $Z$ in $X$ is defined by a finite number of polynomial equations so by standard spreading out arguments there is a $p$-adic field $K$ embedded in $\CC$, with residue field $\FF_q$ over which the arrangement $Z$ is defined, and which has  good reduction.
Denote $f:\mathcal{U}\to \mathcal{X}$ the inclusion of the complement defined over $K$, so that 
$\mathcal{U}\otimes_{K}\CC\cong U$ and $\mathcal{X}\otimes_{K}\CC\cong X$.
By Remark \ref{godemetale} 
we may define
\[(\Aa_{\QQ_\ell}(\mathcal{U}),W):=\Dec\, \Gamma(\mathcal{X}_{\text{pro-ét}},(Rf_*\underline{\QQ}_\ell,\tau^{(2c-1)})).
\]
Now, this filtered algebra is equipped with a Frobenius endomorphism arising from the action of $\mathrm{Gal}(\overline{K}/K)$. When forgetting this action, it is quasi-isomorphic to $\Aa_{\QQ_\ell}(U)$, where $U$ should be understood as the complex analytic space underlying the algebraic variety $U$.

\begin{proposition}\label{Weilalg}
The filtered cdga $(\Aa_{\QQ_\ell}(\mathcal{U}),W)$ satisfies:
\begin{enumerate}
\item Its associated graded is quasi-isomorphic as a cdga to the $E_{2c}$-term of the Leray spectral sequence for $\underline{\QQ}_\ell$ relative to the inclusion $f:\Uu\hookrightarrow \Xx$.
 \item Together with its Frobenius endomorphism, $(\Aa_{\QQ_\ell}(\Uu),W)$ is a Weil algebra.
\end{enumerate} 
\end{proposition}
\begin{proof}
As shown by Weber in \cite{Weber},  the sheaf  $R^kf_*\underline{\QQ}_\ell$ is 0 when $k$ is not divisible by $2c-1$ and $R^{i(2c-1)}f_*\underline{\QQ}_\ell$ decomposes as a direct sum of 
constant sheaves supported by smooth subvarieties of
codimension $ci$. Therefore, by \cite{DeWeil2}, the $\QQ_\ell$-module $H^k(X,R^{i(2c-1)}f_*\underline{\QQ}_\ell)$ is pure of weight $k+2ic$. Since the differentials $d_i$ of the Leray spectral sequence strictly preserve weights, they must vanish for all $i\neq 2c$. In particular, we have \[{}^{\Ll}E_{2}^{*,*}(f,\QQ_\ell)\cong {}^{\Ll}E_{2c}^{*,*}(f,\QQ_\ell).\]
Let $\widetilde{W}:=\Dec^{-1}(W)$.
The quasi-isomorphism of complexes of sheaves
\[Gr^{\tau^{(2c-1)}}_iRf_*\underline{\QQ}_\ell\stackrel{\sim}{\longrightarrow} R^{i(2c-1)}f_*\underline{\QQ}_\ell[-i(2c-1)]\]
(see Lemma \ref{quisspeed}) gives
\[E_1^{-i,j}(\Aa_{\QQ_\ell}(\mathcal{U}),\widetilde W):=H^{j-i}(Gr_{i}^{\widetilde W}\Aa_{\QQ_\ell}(U))\cong {}^{\Ll}E_{2c}^{j-2ic,i(2c-1)}(f,\QQ_\ell)\text{ and }d_1^{\widetilde W}=d_{2c}^{\Ll}.\]
This is just a generalization from $\tau^{(1)}$ to $\tau^{(k)}$ of Example 1.4.8 of \cite{DeHII} relating the hypercohomology spectral sequence of $Rf_*\underline{\QQ}_\ell$ with the filtration $\tau^{(1)}$ relative to the global sections functor, with the Leray spectral sequence.
Since the canonical filtration $\tau^{(k)}$ is a symmetric monoidal functor, this isomorphism is compatible with the algebra structures as well as the Frobenius automorphism.
Lemma \ref{dece0e1} gives a quasi-isomorphism of bigraded complexes
 \[E_0^{-i,j}(\Aa_{\QQ_\ell}(\mathcal{U}),W)\stackrel{\sim}{\lra} E_1^{j-2i,i}(\Aa_{\QQ_\ell}(\mathcal{U}),\widetilde W).\]
All together, we have 
\[E_0^{-i,j+i}(\Aa_{\QQ_\ell}(\mathcal{U}),W)\simeq 
{}^{\Ll}E_{2c}^{i-2c(i-j),(i-j)(2c-1)}(f,\QQ_\ell).
\]
It only remains to note that, since $H^k(X,R^{i(2c-1)}f_*\underline{\QQ}_\ell)$ is pure of weight $k+2ic$, it follows that
$E_1^{-i,j}(\Aa_{\QQ_\ell}(\mathcal{U}),\widetilde W)$ is pure of weight $j$ and hence $E_1^{-i,j}(\Aa_{\QQ_\ell}(\mathcal{U}),W)$ is pure of weight $i$.

\end{proof}

We now prove the main result of this section.

\begin{theorem}\label{maintheorational}
Let $X$ be a smooth complex projective variety and $Z=\{Z_i\}_{1\leq i\leq s}$ an admissible arrangement of codimension $c$ in $X$. Denote 
$U:=X-Z$ and by $f:U\to X$ the inclusion. 
A $\QQ_\ell$-model for $U$ is given by
the cdga  $({}^{\Ll}E_{2c}^{*,*}(f,\QQ_\ell),d_{2c})$.
\end{theorem}
\begin{proof}
A $\QQ_\ell$-model for $U$ is given by $\Aa_{\QQ_\ell}(\Uu)$. By (2) of Proposition \ref{Weilalg}, this algebra carries a Weil structure, so by Proposition \ref{Weilisformal},
 it is homotopically split. This means it is quasi-isomorphic to its associated graded which, by (1) of Proposition \ref{Weilalg}, is given by the $E_{2c}$-page of the Leray spectral sequence.
\end{proof}

\begin{remark}
Theorem \ref{maintheorational} is related to various results in the literature. The case of codimension 1 arrangements is done in \cite{dupontorlik}.
Also, let us remark that Zakharov \cite{zakharovrational} proves more general results on the rational homotopy type of complements of any reasonable arrangement of subvarieties, using mixed Hodge structures.
Our approach through weights in étale cohomology allows us to go beyond the rational case.
\end{remark}

\begin{remark}
A version of Theorem \ref{maintheorational}
over  $\QQ$ instead of $\QQ_\ell$ is also true, and can be proved using a descent argument for homotopy splittings of filtered algebras as in \cite{Morgan} and \cite{CiriciGuillen}.
\end{remark}

\subsection{Examples}
We end this section by reviewing two particular situations where this theorem applies: smooth quasi-projective varieties and configuration spaces. The first example 
recovers Morgan's model \cite{Morgan} for smooth varieties  over $\QQ_\ell$.
The second example recovers the Totaro--Kriz model for configuration spaces of smooth projective varieties \cite{Kriz}. In both cases, our results are functorial.

\begin{example}[Smooth varieties]\label{smoothrational}
Let $U$ be a smooth variety. Consider a good compactification $g:U\hookrightarrow X$ in such a way that the complement $D=X-U$ is a simple normal crossings divisor. 
By Theorem \ref{maintheorational}, a rational model for $U$ is given by the cdga $({}^{\Ll}E_{2}^{*,*},d_{2})$, which we next describe.
We may write
$D=D_1\cup\cdots \cup D_N$ as the union of irreducible smooth varieties meeting transversally. Let $D^{(0)}=X$ and for all $i>0$, denote by $D^{(i)}=\bigsqcup_{|I|=i}D_I$
the disjoint union of all 
$i$-fold intersections 
$D_I:=D_{j_1}\cap\cdots \cap D_{j_i}$ where $I=\{j_1,\cdots,j_i\}$ denotes an ordered subset of $\{1,\cdots,N\}$.
Note that $D^{(i)}$ is smooth projective of dimension $n-i$. We have 
\[{}^{\Ll}E_2^{i,j}=H^{i}(D^{(j)};\QQ_\ell).\]
For $1\leq k\leq i$, let $g_{I,k}:D_I\hookrightarrow D_{I\setminus \{i_k\}}$ denote the inclusion 
and let \[g_{i,k}:=\bigoplus_{|I|=i} g_{I,k}:D^{(i)}\hookrightarrow D^{(i-1)}.\]
We have Gysin maps 
\[(g_{i,k})_!:H^*(D^{(i)},\QQ_\ell)\lra H^{*+2}( D^{(i-1)},\QQ_\ell)\]
and the differential 
\[d_2:{}^{\Ll}E_2^{i,j}=H^{i}(D^{(j)};\QQ_\ell)\lra {}^{\Ll}E_2^{i+2,j-1}=H^{i+2}(D^{(j-1)};\QQ_\ell)\]
is given by
\[d_2:=\sum_{k=1}^i (-1)^{k} (g_{i,k})_{!}.\]
The algebra structure on ${}^{\Ll}E_2^{*,*}$ is induced by the combinatorial restriction morphisms
\[g_{i}^*=\sum_{k=1}^i (-1)^{k-1} (g_{i,k})^{*}\] together with the cup product of $H^*(D^{(i)};\ZZ_\ell)$, for $i>0$.

 Note that if $X$ is a smooth projective variety, we may take $U=X$ and $D=\emptyset$.
 Therefore in this case ${}^{\Ll}E_{2}^{*,*}$ is just the cohomology ring of $X$ with $\QQ_\ell$-coefficients, with trivial differential, recovering formality over $\QQ_\ell$ for smooth projective varieties (see \cite{DeWeil2}).
Note that such models are compatible with maps of smooth quasi-projective varieties. Indeed, given a map $f:U\to U'$ one may always find compactifications of $U$ and $U'$ and an extension of $f$  to a regular map between the compactifications, so one gets an induced map of Leray spectral sequences.
\end{example}

\begin{example}[Configuration spaces]\label{rationalconfexample}
Given a smooth complex projective variety $X$ of complex dimension $m$, let 
 \[Z_{ij}:=\{(x_1,\cdots,x_n)\in X^n; x_i=x_j\}.\] Then $Z=\{Z_{ij}\}$ is an admissible arrangement of codimension $m$ in $X^n$, and  \[U:=X^n-Z=\mathrm{Conf}_n(X).\] Denote by $f:U\hookrightarrow X^n$ the inclusion. By Theorem \ref{maintheorational}, a model for $\mathrm{Conf}_n(X)$ is given by the cdga $({}^{\Ll}E_{2m}^{*,*},d_{2m})$. This is described in \cite{Totaro} as follows.
Consider the bigraded commutative $\QQ_\ell$-algebra 
\[H^*(X^n,\QQ_\ell)[x_{ij}]\,;\, 1\leq i,j\leq n,\, i\neq j\]
where $H^i(X^n,\QQ_\ell)$ has bidegree $(i,0)$ and $x_{ij}$ has bidegree $(0,2m-1)$. Then ${}^{\Ll}E_{2m}^{*,*}$ is identified with the quotient of this algebra by the relations 
\begin{align*}
 x_{ij}&=x_{ji}\\
 x_{ij}x_{jk}+x_{jk}x_{ki}+x_{ki}x_{ij}&=0\text{ for }i,j,k\text{ distinct}\\
 p_i^*(x) x_{ij}&=p_j^*(x)x_{ij}\text{ for }i\neq j, x\in H^*(X)
\end{align*}
Here $p_i:X^n\to X$ denotes the projection to the $i$-th component. The differential $d_{2m}$ is given by
\[d(x_{ij})=p_{ij}^* \Delta,\]
where $p_{ij}:X^n\to X\times X$ is the projection to the $i$th and $j$th components and $\Delta\in H^{2m}(X^2)$ denotes the class of the diagonal.

Observe that this model is functorial with respect to injective maps of algebraic varieties. Such maps induce maps at the level of configuration spaces. In particular, we obtain equivariance of configuration space models with respect to any discrete group of automorphisms of the variety $X$.
\end{example}

\section{Tame homotopy theory}\label{sectiontame}

Tame homotopy theory is a refinement of rational homotopy theory introduced by Dwyer in \cite{Dwyer}. It captures the ``easy'' part of integral homotopy theory, namely the part that is insensitive to Steenrod operations. Precisely, a map $f:X\to Y$ between based $(r-1)$-connected spaces is called a \textit{tame weak equivalence} if it induces an isomorphism
\[\pi_{r+k}(X)\otimes R_k\cong\pi_{r+k}(Y)\otimes R_k\]
where $R_k$ denotes a subring of $\QQ$ in which the prime numbers $p$ such that $2p-3\leq k$ are invertible. This notion depends on the choice of the sequence of rings $R_k$. The main theorem of \cite{Dwyer} is that the homotopy theory of spaces up to tame weak equivalences can be encoded fully faithfully into the homotopy theory of differential graded Lie algebras over the integers whose cohomology satisfies some divisibility condition.
Alternatively, Anick in \cite{Anick} showed that tame homotopy types can be modeled algebraically by Hopf algebras up to homotopy. A refinement of this approach was developed by Hess \cite{Hess}, introducing mild homotopy types.
The recent PhD thesis of Haoqing Wu \cite{haoqinghopf} contains an $\infty$‑categorical upgrade of Anick's approach.
A dual way of encoding tame homotopy types was introduced by Cenkl-Porter \cite{CePo},
in terms of filtered commutative dg-algebras (see also \cite{Sch1}, \cite{scheererhomotopie}).
A more recent and streamlined exposition of this dual approach appears in \cite{Hanke}, where Hanke uses tame algebraic models in the context of transformation group theory.

The goal of this section is to revisit some of the foundations of tame homotopy theory, adapted to our situation of interest. We work with coefficients in $\ZZ_\ell$, for a prime number $\ell$, and develop the theory in terms of two integer parameters $r\geq 1$ and $s\geq 0$ accounting for the connectivity and a tame parameter respectively.
When $s=0$ the theory reduces to rational homotopy, and increasing $s$ restores more of the integral homotopy type of the space. In the simply-connected case, we prove that the tame homotopy type of a space is encoded faithfully in its $E_\infty$-algebra of cochains over $\ZZ_\ell$, considered up to a notion of \textit{tame quasi-isomorphism} (see Theorem \ref{theo : tame mandell} below). In contrast with Dwyer's tame homotopy theory, our approach can also be applied to non-simply-connected spaces, giving in this case information on the fundamental group.

\subsection{Tame homotopy types}
Fix $\ell$ a prime number. An \textit{$\ell$-complete space} is a space that is local with respect to homology with coefficients in $\mathbb{F}_\ell$. Any space $X$ admits an \textit{$\ell$-completion} $X^\wedge_\ell$ with a map
\[X\to X^\wedge_\ell\]
which is homotopy initial among maps from $X$ to an $\ell$-complete space. If $X$ is simply-connected and of finite type then the above map is simply given by algebraic $\ell$-completion on homotopy groups. We refer to \cite{BoKa} for details about this construction.

For an integer $n\geq 0$, we shall denote by $t_{\leq n}$ the truncation functor from spaces to $n$-truncated spaces (i.e. spaces with trivial homotopy groups above degree $n$). It is defined so that there is a natural map
\[X\to t_{\leq n} X\]
which induces an isomorphism on homotopy groups of degree $\leq n$ and the zero map in degree $>n$.

\begin{lemma}\label{CWtruncate}
Let $X$ and $Y$ be two simply-connected CW-complexes of dimension $\leq n$. Then, there is a weak equivalence $t_{\leq n}X\simeq t_{\leq n}Y$ if and only if $X\simeq Y$.
\end{lemma}

\begin{proof}
The reverse implication is obvious. In order to prove the direct implication, we shall first prove the more general fact that a map $f:t_{\leq n}X\to t_{\leq n}Y$ admits a unique lift up to homotopy to a map $\tilde{f}:X\to Y$. Assuming this for the moment, our weak equivalence $f:t_{\leq n}X\xrightarrow{\simeq} t_{\leq n}Y$ and its  homotopy inverse $g:t_{\leq n}Y\xrightarrow{\simeq} t_{\leq n}X$ admit lifts $\tilde{f}:X\to Y$ and $\tilde{g}:Y\to X$. Now, the composite $\tilde{f}\circ\tilde{g}$ is a lift of the identity map of $t_{\leq n}Y$ and so must be homotopic to $\id_Y$ and similarly $\tilde{g}\circ\tilde{f}$ is homotopic to $\id_X$.

It remains to prove the claim. We can proceed by induction on the Postnikov tower of $Y$. Assume that a map $X\to t_{\leq k}Y$ lifting the composite
\[X\to t_{\leq n}X\xrightarrow{f}t_{\leq n}Y\]
has been constructed for some $k\geq n$. Then the obstruction to lifting it further to $t_{\leq k+1}Y$ lives in $H^{k+2}(X,\pi_{k+1}Y)$ and the set of all lifts is acted on transitively by the group $H^{k+1}(X,\pi_{k+1}Y)$. Both of these groups are zero by the dimension assumption of $X$ which shows that the lift exists and is unique.
\end{proof}

Throughout, $r\geq 1$  and $s\geq 0$ are integers, where $r$ specifies a connectivity bound and $s$ serves as a tame parameter.

\begin{definition}
Let $f:X\to Y$ be a map between based $(r-1)$-connected $\ell$-complete spaces. We say that $f$ is an \textit{$(r,s)$-tame weak equivalence} if the map $\pi_{r+k}(f)$ is an isomorphism for $k< s$ and if $\pi_{r+k}(f)\otimes\QQ$ is an isomorphism for all values of $k$.
\end{definition}

Let $\cat{S}_r$ be the category of based $(r-1)$-connected $\ell$-complete spaces. We consider the category $\cat{S}_{r}\st$ defined by the following pullback square of $\infty$-categories
\[\xymatrix{
\cat{S}_r\st\ar[r]\ar[d]&t_{\leq r+s-1}\cat{S}_r\ar[d]\\
L_{\QQ}\cat{S}_r\ar[r]&t_{\leq r+s-1}L_{\QQ}\cat{S}_r
}
\] 
where the two horizontal maps are
localization functors given by truncation and the vertical maps are rationalization. Note that truncation and rationalization commute.
Concretely, an object of this category is a triple $(X,Y,\alpha)$ with $X$ a $t_{\leq r+s-1}$-truncated space, $Y$ a rational space and $\alpha$ the data of a weak equivalence $L_{\QQ}X\simeq t_{\leq r+s-1}Y$.

\begin{definition}
We call $\cat{S}_r\st$ the $\infty$-category of \textit{$(r,s)$-tame homotopy types}. We denote by $X\mapsto X\st$ the functor
\[\cat{S}_r\to\cat{S}_r\st\]
given by
\[X\mapsto (t_{\leq r+s-1}X,X_{\QQ},\id),\]
and refer to it as the \textit{$s$-tamification functor}.
\end{definition}

Note that this category and the tamification functor depend on the fixed prime number $\ell$ although we do not include it in the notation. Note also that there are forgetful functors
\[\cat{S}_r\st\to \cat{S}_r\smt\]
and for $s=0$ we recover rationalization, so
$\cat{S}_r^{0-\text{tame}}\simeq L_{\QQ}\cat{S}_r$. Observe also that a map $f:X\to Y$ between $(r-1)$-connected $\ell$-complete homotopy types is an $(r,s)$-tame weak equivalence if and only if its image in $\cat{S}_r\st$ is a weak equivalence.

\medskip

We will use the following result about the loop space construction.

\begin{proposition}\label{prop : product of EM}
Let $r\geq 2$. Let $s$ be a non-negative integer satisfying $s\leq 2\ell-3$ if $r\geq 3$ or $s\leq \ell-2$ if $r=2$. Let $X$ be an $(r-1)$-connected pointed space with finite type $\ZZ_\ell$-cohomology. Then, $\Omega X$ is $(r-1,s)$-tame weakly equivalent to a product of Eilenberg--MacLane spaces.
\end{proposition}

\begin{proof}
What we mean precisely is that $t_{\leq r+s-2}\Omega X$ and $(\Omega X)_{\QQ}$ are products of Eilenberg--MacLane spaces and that the weak equivalence
\[(t_{\leq r+s-2}\Omega X)_{\QQ}\simeq t_{\leq r+s-2}( (\Omega X)_{\QQ})\]
is compatible with these decompositions. This can be expressed in terms of $k$-invariants. We are claiming that the $n$-th $k$-invariant of $\Omega X$
\[\alpha_n\in H^{r+n}(t_{\leq r+n-2}\Omega X,\pi_{r+n}(X))\]
is zero when $n<s$ and that
\[\alpha_n\otimes\QQ\in H^{r+n}(t_{\leq r+n-2}\Omega X,\pi_{r+n}(X)\otimes\QQ)\]
is always zero. This statement is essentially the content of the proof of \cite[Proposition 3]{soule} except that Soulé does not treat the case $r=2$.

Let $n<s$. As in \cite[Proposition 3]{soule}, the class $\alpha_n$ must be primitive and decomposable by the induction hypothesis. If this class is non-zero and $n<s$, this implies that this class is an $\ell$-th power which means that
\[r+n\geq \ell (r-1)\]
since $\ell (r-1)$ is the lowest degree in which $\ell$-th powers can occur. Therefore
\[s> (\ell-1)r-r.\]
This contradicts our assumption on $s$.

If $n\geq s$, then by the induction hypothesis there is an isomorphism
\[H^*(t_{\leq r+n-2}\Omega X,\pi_{r+n}(X)\otimes\QQ)\cong P\otimes \pi_{r+n}(X)\]
where $P=H^*(t_{\leq r+n-2}\Omega X,\QQ_\ell)$ is a graded polynomial algebra whose generators have degree in the interval $[r-1,r+n-2]$. Since $\alpha_n$ is a primitive class of degree $r+n$ in this Hopf algebra it must be zero.
\end{proof}

\subsection{Filtered tame cochains}We next introduce the $\infty$-category $\cat{FCh}_r\st$ of tame filtered complexes.
Denote by $\cat{FCh}$ the $\infty$-category of filtered cochain complexes. This is the $\infty$-category of diagrams from the poset $\mathbb{N}$ to the $\infty$-category of cochain complexes. Note that, since our filtrations are bounded below, a map in this $\infty$-category is a weak equivalence if and only if it induces an isomorphism on the $E_1$-page of the associated spectral sequence.

We write an object of $\cat{FCh}$ as a pair $(C,W)$. The value of our functor at $i\in\NN$ is denoted $W_iC$, so that $(C,W)$ is a short-hand notation for the diagram
\[W_0C\to W_1C\to\ldots\]

\begin{definition}
A filtered cochain complex $(C,W)$ is said to be \textit{$(r-1)$-connected} if $W_i C$ is acyclic when $i<r$. We denote by $\cat{FCh}_r$ the $\infty$-category of $(r-1)$-connected filtered cochain complexes.
\end{definition}

\begin{definition}\label{defi : s-tame weak equivalence}
A map $f:C\to D$ of $(r-1)$-connected filtered cochain complexes over $\ZZ_\ell$
is called an \textit{$(r,s)$-tame quasi-isomorphism} if it induces
\begin{enumerate}
\item an isomorphism
\[H^{*}(W_iC)\cong H^{*}(W_iD)\]
for $i<r+s$,
\item an isomorphism
\[H^*(W_iC)\otimes\QQ_\ell\to H^*(W_iD)\otimes\QQ_\ell\]
for all values of $i$.
\end{enumerate}
\end{definition}

We denote by $\cat{FCh}_r\st$ the localization of the $\infty$-category of $(r-1)$-connected filtered cochain complexes at the $(r,s)$-tame quasi-isomorphisms. This is a non-unital symmetric monoidal $\infty$-category.

\begin{remark}\label{rem: tamification}
The $\infty$-category $\cat{FCh}\st_r$ is a localization of the $\infty$-category  of  $(r-1)$-connected filtered complexes with respect to filtered quasi-isomorphisms. The localization functor is very explicit, it is given by \[C\mapsto T_{(r,s)}C,\] where
\[
W_kT_{s}C:=\left\{
\begin{array}{ll}
W_kC\;&\text{ if }\,k <  r+s\\
     W_kC\otimes\QQ_\ell \;&\text{ if }\,k\geq  r+s
\end{array}\right..
\]
The $\infty$-category $\cat{FCh}\st_r$ is thus equivalent to the essential image of $T_{(r,s)}$ in the usual $\infty$-category of filtered complexes.
\end{remark}

\begin{proposition}\label{prop : third definition} A map of $(r-1)$-connected filtered cochain complexes
$f:C\to D$
is an $(r,s)$-tame quasi-isomorphism if and only if the induced map
\[\FF_\ell\otimes Gr_i^W(f):\FF_\ell\otimes Gr_i^WC\to \FF_\ell\otimes Gr_i^WD\]
induces an isomorphism on cohomology for $i< r+s$ and the map
\[\QQ_\ell\otimes Gr_i^W(f):\QQ_\ell\otimes Gr_i^WC\to \QQ_\ell\otimes Gr_i^WD\]
induces an isomorphism on cohomology for any $i$.
\end{proposition}

\begin{proof}
This follows easily from the standard fact that, a map $f:V\to W$ of complexes of $\ZZ_\ell$-modules is a quasi-isomorphism if and only if $f\otimes \FF_\ell$ and $f\otimes\QQ_\ell$ are quasi-isomorphisms.
\end{proof}

Denote by $\tau$ the canonical filtration of a cochain complex. This corresponds to the filtration $\tau^{(1)}$ of Definition \ref{Ciricifiltration}.

\begin{proposition}\label{prop: notused}
Let $X$ and $Y$ be two $\ell$-complete $(r-1)$-connected  based spaces. Let $f:X\to Y$ be an $(r,s)$-tame weak equivalence. Then the induced map
\[(\tilde{C}^*(X,\ZZ_\ell),\tau)\to (\tilde{C}^*(Y,\ZZ_\ell),\tau)\]
is an $(r,s)$-tame quasi-isomorphism.
\end{proposition}

\begin{proof}
Since $f$ is an $(r,s)$-tame weak equivalence, we know that the following two facts hold:
\begin{enumerate}
\item The map $f$ is a rational weak equivalence.
\item The map $\tau_{r+s-1}(f)$ is a weak equivalence.
\end{enumerate}
The first fact implies that $f$ induces an isomorphism in $H^*(-,\QQ_\ell)$. 

The second fact implies that $\tilde{H}_*(f,\ZZ_\ell)$ is an isomorphism in degrees $*<r+s$. In order to see this, consider the following diagram
\[
\xymatrix{
\cat{S}_r\ar[r]^-{t_{\leq r+s-1}}\ar[d]_{\tilde{C}_*(-,\ZZ_\ell)}&t_{\leq r+s-1}\cat{S}_{r}\ar[d]^{t_{\leq r+s-1}\tilde{C}_*(-\ZZ_\ell)}\\
\cat{Ch}_{r}\ar[r]_-{t_{\leq r+s-1}}&t_{\leq r+s-1}\cat{Ch}_{r}
}
\]
where the $r$-subscript indicates that we restrict to $(r-1)$-connected objects. This is a diagram of left adjoints. It commutes since the corresponding diagram of right adjoints commutes. To see this, observe that the right adjoint of the two horizontal functors are simply the inclusions and the right adjoint of the vertical functors are the composite
\[\cat{Ch}\simeq \cat{Mod}_{\ZZ_\ell}(\cat{Sp})\xrightarrow{\text{forget}}\cat{Sp}\xrightarrow{\Omega^\infty}\cat{S}.\]
But the commutativity of this diagram together with fact (2) above implies that $\tilde{C}_*(f,\ZZ_\ell)$ is an isomorphism in homology in degree $<r+s$ as desired. 

Finally, by the universal coefficient theorem, if $\tilde{H}_*(f,\ZZ_\ell)$ is an isomorphism in degrees $*<r+s$ so it $\tilde{H}^*(f,\ZZ_\ell)$.
\end{proof}

\begin{remark}
For $r\geq 2$, we will prove, in Theorem \ref{theo : tame mandell}, a partial converse of the above result.
\end{remark}

\subsection{Canonization of filtered complexes}\label{subsection : different filtration}

This short section can be skipped on a first read. Its point is to prove Proposition \ref{prop : different filtration} that will be useful in later sections of this paper. As we shall see, the tame homotopy type of a space is determined, up to tame quasi-isomorphism, by its cochain algebra equipped with the canonical filtration. In practice, certain constructions of interest produce a filtered algebra carrying a non-canonical filtration. Proposition \ref{prop : different filtration} establishes conditions under which we can still recover the tame homotopy type from such an algebra.

Let $(C,W)$ be a filtered complex over $\ZZ_\ell$. Then we may form a new filtered complex $(C,\tau W)$ defined by
\[\tau W_kC:=\tau_{k}W_k C\]

\begin{lemma}The functor $(C,W)\mapsto (C,\tau W)$ descends to a symmetric monoidal endofunctor of $\cat{FCh}_r\st$.
\end{lemma}

\begin{proof}
This follows from the obvious natural quasi-isomorphism
\[\tau_{k}(C\otimes \QQ_\ell)\simeq (\tau_{k}C)\otimes\QQ_\ell\]
for any cochain complex of $\ZZ_\ell$-modules $C$.
\end{proof}

For any functor $\sigma:(\mathbb{N},\leq)\to(\mathbb{N},\leq)$ (i.e. a non-decreasing map), we get an induced endofunctor on the category of filtered complexes
\[(C,W)\mapsto (C,\sigma^{-1}W)\]
where the filtration $\sigma^{-1}W$ is simply the composition
\[(\mathbb{N},\leq)\to (\mathbb{N},\leq)\xrightarrow{W}\cat{Ch}.\]
If $\sigma$ is moreover lax monoidal, then the resulting endofunctor is also lax monoidal.

\begin{remark}
Observe that a lax monoidal map from $(\mathbb{N},\leq)$ to itself is simply a non-decreasing superadditive map, i.e. a map satisfying
\[\sigma(m+n)\geq \sigma(m)+\sigma(n)\]
for any $m$ and $n$ in $\mathbb{N}$. Our main example of such a map is
\[\sigma(m)=\lfloor \alpha m\rfloor\]
for $\alpha$ a positive rational number.
\end{remark}

\begin{proposition}
Let $A$ be a complex equipped with a filtration $W$ such that $(A,W)$ is an $(r-1)$-connected filtered complex. Let $\sigma:\mathbb{N}\to\mathbb{N}$ be a non-decreasing superadditive map. Assume that for each $i$, the canonical map
\[W_{\sigma(i)} A\to A\]
induces an isomorphism in cohomology in degree $\leq i$. Then, the filtered complex $(A,W)$ up to $(r,s)$-tame quasi-isomorphism determines the $(r',s')$-tame homotopy type of $(A,\tau)$ with
\[s'=\mathrm{max}\{n\in\mathbb{N},\sigma(n)\leq s\}\quad\text{ and }\quad r'=\mathrm{min}\{n\in\mathbb{N},\sigma(n)\geq r\}.\]
\end{proposition}

\begin{proposition}\label{prop : different filtration}
Let $r$ and $s$ be two non-negative integers. Let $\sigma:\mathbb{N}\to\mathbb{N}$ be a non-decreasing superadditive map. Let $r'$ and $s'$ be defined by
\[s'=\mathrm{max}\{n\in\mathbb{N},\sigma(n)\leq s\}\quad\text{ and }\quad r'=\mathrm{min}\{n\in\mathbb{N},\sigma(n)\geq r\}.\]
Then there exists a lax symmetric monoidal functor
\[\cat{FCh}_r\st\to \cat{FCh}_{r'}^{s'\textrm{-tame}}\]
which is naturally equivalent to $(A,W)\mapsto (A,\tau)$ when we apply it to a complex $A$ equipped with a filtration $W$ which is such that the canonical map
\[W_{\sigma(i)} A\to A\]
induces an isomorphism in cohomology in degree $\leq i$.
\end{proposition}

\begin{remark}
This proposition is likely to be hard to parse at first read. It is simply saying that, for a complex $A$ equipped with a filtration $W$ which is ``faster'' than the canonical filtration up to rescaling by $\sigma$, the $(r,s)$-tame homotopy type of $(A,W)$ determines the $(r',s')$-tame homotopy type of $(A,\tau)$. Moreover any kind of algebraic structure is preserved in the process.
\end{remark}

\begin{proof}
This functor is given by
\[(A,W)\mapsto (A,\tau\sigma^{-1}W).\]
It is clearly lax monoidal and sends $(r,s)$-tame quasi-isomorphisms to $(r',s')$-tame quasi-isomorphisms by the previous lemma. Finally, when $(A,W)$ satisfies the condition that the canonical map
\[W_{\sigma(i)} A\to A\]
induces an isomorphism in cohomology in degree $\leq i$, there is a filtered quasi-isomorphism (and hence an $(r',s')$-tame quasi-isomorphism)
\[(A,\tau\sigma^{-1}W)\simeq (A,\tau).\qedhere\]
\end{proof}

\subsection{Filtered bar construction}
We relate two filtrations on the bar construction of an algebra: the one induced by the canonical filtration and the skeletal filtration. This is then used to characterize the behaviour of tame quasi-isomorphisms under the bar construction.

Given an augmented dg-algebra $A$ over a commutative ring $R$, one can form its bar construction $\mathrm{Bar}(A)$. This is the realization of the simplicial object
\[[n]\mapsto A^{\otimes n}\]
which computes the relative derived tensor product $R\otimes_AR$. If $A$ happens to be an augmented $E_\infty$-algebra over $R$, then the whole simplicial diagram is a diagram of $E_\infty$-algebras and the bar construction inherits an $E_\infty$-structure.

Now assume that $A$ is equipped with a compatible filtration $W$. Then the simplicial object becomes a simplicial object in filtered complexes. So the bar construction itself inherits a filtration. Moreover, since the associated graded functor is strong monoidal and commutes with colimits, we see that the $E_1$-page of the resulting spectral sequence is given by
\[E_1^{-*,j}=H^{j-*}(\mathrm{Bar}(Gr^W_*A)).\]

More generally, let $C_\bullet$ be a simplicial object in connective cochain complexes (i.e., for each $n$, the complex $C_n$ does not have cohomology in negative degree). Consider the following two filtrations on 
\[|C_\bullet|:=\mathrm{hocolim}_{\Delta^\op}C_\bullet.\]

\begin{enumerate}
\item The filtration induced by the canonical filtration
\[W_p|C_\bullet|:= |\tau_p C_{\bullet}|\]
\item The \textit{skeletal filtration}
\[Sk_p|C_\bullet|:=\mathrm{hocolim}_{\Delta_{\leq p}^{\op}}C_\bullet\]
\end{enumerate}
We may also form these two filtrations if $C_\bullet$ is merely a semi-simplicial object in connective cochain complexes, i.e. a contravariant diagram on $\Delta_{inj}$ the category of non-empty finite totally ordered sets and injective maps. The following shows that the two filtrations defined above are related by the décalage functor defined in \ref{defidec} (see also \cite[Section 6]{Levine} for a cosimplicial version).

\begin{lemma}\label{lemm : two filtrations}
Let $C_\bullet$ be a simplicial or semi-simplicial object in connective cochain complexes. Then we have  $\Dec (Sk)=W$.
\end{lemma}

\begin{proof}
Before going into the proof, we recall a classical computation of the associated graded of the skeletal filtration. In the simplicial case, then the homotopy cofiber of $Sk_{p-1}|C_\bullet|\to Sk_{p}|C_\bullet|$ is given by the $p$-fold suspension of the total homotopy cofiber of a $p$-dimensional cube  built from the values $C_k$ with $k\leq p$ and the codegeneracy maps. In the semi-simplicial case, this homotopy cofiber is simply the $p$-fold suspension of $C_p$.

Now, we follow closely the proof of \cite[Proposition 9.2]{antieauspectral} in cohomological grading conventions. We do the proof in the simplicial case. The semi-simplicial case being completely similar. Consider the bifiltered complex
\[D_{a,b}=Sk_a|\tau_b C|\]
then the associated graded piece $D_{a,b}/D_{a-1,b}$ is the $a$-fold suspension of the total cofiber of a cube of cochain complexes whose cohomology is concentrated in the interval $[0,b]$. It follows that the cohomology of $D_{a,b}/D_{a-1,b}$ is concentrated in degrees at most $b-a$. Therefore the obvious map
\[D_{*,b}\to Sk_*\]
factors through $\tau^{B}_{b}Sk_*$ where $\tau^B$ denotes the filtration corresponding to the Beilinson t-structure. The induced map $D_{*,b}\to \tau^{B}_{b}Sk_*$  is an equivalence, since it is one on associated graded as shown by the following computation
\[D_{a,b}/D_{a-1,b}\simeq \mathrm{tcofib}_{s\leq a}(s\mapsto \tau_{b}C_s[a])\simeq (\tau_{b-a}\mathrm{tcofib}_{s\leq a}(s\mapsto C_s))[a]\]
\[\simeq \tau_{b-a}Gr_a^{Sk}|C|\simeq Gr_a\tau_{b}^{B}Sk_*\]
It follows that $D_{*,b}\simeq \tau_{b}^BSk_*$ which implies that $\Dec (Sk)\simeq W$.
\end{proof}

As a corollary, we obtain the following result.

\begin{corollary}\label{coro : E2 page of W filtration}
For $k\geq 1$, there is an isomorphism
\[E_{k}^{-i,j}(|X_{\bullet}|,W)\cong E_{k+1}^{j-2i,i}(|X_\bullet|,Sk). \]

In particular, in the semi-simplicial case, we get an isomorphism
\[E_1^{-i,j}(|X_\bullet|,W)\cong  H^j(X_i)\]

and in the simplicial case, we get an isomorphism
\[E_1^{-i,j}(|X_\bullet|,W)\cong H^j(X_i)/D(H^j(X_i))\]
where $D(H^j(X_i))$ is the subgroup generated by the images of the degeneracy maps $H^j(X_{i-1})\to H^j(X_i)$.

\end{corollary}

\begin{proposition}\label{prop : bar construction}
Let $r\geq 2$. Let $f:A\to B$ be a map between two non-unital filtered $(r-1)$-connected algebras over $\ZZ_\ell$. If the induced map $(A,\tau)\to (B,\tau)$ is an $(r,s)$-tame quasi-isomorphism, then the map
\[\mathrm{Bar}(f):(\mathrm{Bar}(A),\tau)\to (\mathrm{Bar}(B),\tau)\]
is an $(r-1,s-\lfloor\frac{s+1}{r}\rfloor)$-tame quasi-isomorphism.
\end{proposition}

\begin{proof}
We know that $\mathrm{Bar}(A)$ is the colimit of the usual semi-simplicial diagram
\[[k]\mapsto A^{\otimes k}\]
where the face maps are induced by multiplication of two adjacent copies of $A$ or the zero map for the two extreme faces. We denote this semi-simplicial object $\mathrm{Bar}_{\bullet}(A)$ and we denote $\mathrm{Bar}(A)$ its homotopy colimit. Applying the canonical filtration levelwise, we get a filtered semi-simplicial object $\tau_*\mathrm{Bar}_{\bullet}(A)$. Taking colimit in the semi-simplicial direction, we obtain a filtration on $\mathrm{Bar}(A)$ that we denote $W$. Explicitly, we have
\[W_n\mathrm{Bar}(A):=\mathrm{colim}_k\tau_n\mathrm{Bar}_k(A).\]
This construction is of course functorial in $A$. Clearly, our map $f:A\to B$ induces an
$(r,s)$-tame quasi-isomorphism
\[(\mathrm{Bar}(A),W)\to(\mathrm{Bar}(B),W).\]

We consider the spectral sequence associated to $W$.
By Corollary \ref{coro : E2 page of W filtration}, it reads
\[E_1^{-i,j}(W)=E_2^{j-2i,i}(Sk)\implies H^{j-i}(\mathrm{Bar}(A)).\]
From the connectivity of $A$ and using $E_1^{-i,j}(Sk)=H^{j}(A^{\otimes i})$ we have that
$E_1^{-i,j}(Sk)=0$ when $j<ri$. Therefore we have that
\[E_1^{-i,k+i}(W)=E_2^{k-i,i}(Sk)=0\text{ for }i>{r\over r-1}k.\]
It follows that this is also the case for the same bidegrees at the $E_\infty$-page.
The cohomology group $H^k(\mathrm{Bar}(A))$ has a filtration induced by $W$ whose associated graded is $\oplus_iE_\infty^{-i,k+i}$.
From this we deduce that
\[W_{\lfloor{r\over r-1}k\rfloor}\mathrm{Bar}(A)\to \mathrm{Bar}(A)\]
induces an isomorphism in cohomology up to degree $k$. In order to conclude the proof we simply apply Proposition \ref{prop : different filtration} with $\sigma(k)=\lfloor{r\over r-1}k\rfloor$.
\end{proof}

\subsection{A Mandell-type Theorem in tame homotopy theory}

Recall that if $X$ and $Y$ are nilpotent spaces of finite type, Mandell's Theorem \cite{mandellcochains} shows that $X$ and $Y$ are weakly homotopy equivalent if and only if their singular cochains $C^*(X,\ZZ)$ and $C^*(Y, \ZZ)$ are quasi-isomorphic as $E_\infty$-algebras over $\ZZ_\ell$. In the tame setting, the easy implication is proven in Proposition \ref{prop: notused}. In this section we prove a converse implication, giving a Mandell-type Theorem in tame homotopy theory over $\ZZ_\ell$.

Let us first recall the following classical fact about cohomology of Eilenberg--MacLane spaces.

\begin{theorem}\label{theo : cohomology of EM spaces}
Let $n\geq 2$. Let us denote by $\iota_n$ the fundamental class of $K(\ZZ,n)$ in $H^n(K(\ZZ,n),\QQ)$, or in $H^n(K(\ZZ,n),\FF_\ell)$. Likewise, let us denote by $\iota_n$ the fundamental class of $K(\ZZ/\ell^k,n)$ in $H^n(K(\ZZ/\ell^k,n),\FF_\ell)$. Then
\begin{enumerate}
\item The canonical map
\[\FF_\ell[\iota_n]\to H^*(K(\ZZ,n),\FF_\ell)\]
is an isomorphism of graded commutative rings in degree $\leq n+2\ell-3$.
\item Denoting $\beta \iota_n$ the $k$-th Bockstein of $\iota_n$, the canonical map
\[\FF_\ell[\iota_n,\beta\iota_n]\to H^*(K(\ZZ/\ell^k,n),\FF_\ell)\]
is an isomorphism of graded commutative rings in degree $\leq n+2\ell-3$.
\item The canonical map
\[\QQ_\ell[\iota_n]\to H^*(K(\ZZ,n),\QQ_\ell)\]
is an isomorphism of graded commutative rings.
\end{enumerate}
\end{theorem}

\begin{proof}
This comes from a more precise theorem, namely that $H^*(K(\ZZ,n),\FF_\ell)$ is a polynomial algebra with generators of the form $Q\iota_n$ with $Q$ an explicit composite of Steenrod operations (see for example \cite[Theorem 3.4]{tamanoi}). The important fact is that, in the first case, the smallest degree of $Q$ allowed besides $Q=1$ is $Q=P_1$ which is of degree $2\ell-2$ and in the second case, the only other $Q$ that is allowed is $Q_0$ which is the $k$-th Bockstein operator. The rational computation is very classical.
\end{proof}

For $m$ and $n$ positive integers and $A$ a finitely generated $\ZZ_\ell$-module we write $F(A,n)$ for the filtered  $E_\infty$-algebra given by
\[F(A,n)=\Sym((\RR\Hom(A[n],\ZZ_\ell),\tau)).\]

Now, we construct a map of filtered algebras
\[f:F(A,n)\to (\tilde{C}^*(K(A,n),\ZZ_\ell),\tau)\]
that is functorial in $A$. For this, we start from the canonical map of chain complexes
\[\tilde{C}_*(K(A,n),\ZZ_\ell)\to A[n]\]
that can be seen as truncation in degree $\leq n$. Dualizing this map, we get a map of cochain complexes
\[\RR\Hom(A[n],\ZZ_\ell)\to \tilde{C}^*(K(A,n),\ZZ_\ell).\]
Taking canonical filtration on both sides, we get a map
\[(\RR\Hom(A[n],\ZZ_\ell),\tau)\to (\tilde{C}^*(K(A,n),\ZZ_\ell),\tau).\]
Using the adjunction between filtered complexes and filtered algebras, such a map amounts to a map of filtered algebras
\[f:F(A,n)\to (\tilde{C}^*(K(A,n),\ZZ_\ell),\tau).\]

\begin{lemma}\label{lemm : case of EM spaces}
For $n\geq r\geq 2$ and $s\leq 2\ell-3$, the map
\[f:F(A,n)\to (\tilde{C}^{*}(K(A,n),\ZZ_\ell),\tau)\]
is an $(r,s+1)$-tame quasi-isomorphism. 
\end{lemma}

\begin{proof}
Since $A$ splits as a product of copies of $\ZZ_\ell$ and $\ZZ/\ell^k$ and since both the source and the target send finite products in the $A$ variable to tensor products, we may assume without loss of generality that $A$ is either $\ZZ_\ell$ or $\ZZ/\ell^k$. We do the case of $A=\ZZ/\ell^k$. The other case is similar and easier. According to Proposition \ref{prop : third definition}, it is enough to prove that $f$ induces an isomorphism
\[E_1^{-i,*}(F(A,n)\otimes \FF_\ell)\to E_1^{-i,*}(\tilde{C}^{*}(K(A,n),\FF_\ell))\]
for $i\leq r+s$, and an isomorphism
\[E_1^{-i,*}(F(A,n)\otimes \QQ_\ell)\to E_1^{-i,*}(\tilde{C}^{*}(K(A,n),\QQ_\ell))\]
for any value of $i$. The rational claim is trivial since both groups are actually zero. 

In order to deal with the $\FF_\ell$-case, we consider the commutative triangle of filtered complexes
\[\xymatrix{
((\RR\Hom(A[n],\ZZ_\ell),\tau)\ar[rd]\ar[r]&F(A,n)\ar[d]\\
 &(\tilde{C}^*(K(A,n),\ZZ_\ell),\tau)
}
\]
We can tensor it with $\FF_\ell$ and consider the resulting triangle of filtered complexes of $\FF_\ell$-modules
\[\xymatrix{
((\RR\Hom(A[n],\FF_\ell),\tau_\ell)\ar[rd]\ar[r]&(F(A,n)\otimes\FF_\ell,W)\ar[d]\\
 &(\tilde{C}^*(K(A,n),\FF_\ell),\tau_\ell)
}
\]
where $\tau_\ell$ is our notation for the filtration induced by the canonical filtration after tensoring with $\FF_\ell$ (beware that this is no longer necessarily the canonical filtration). We have $\RR\Hom(A[n],\ZZ_\ell)\simeq A[-(n+1)]$ and the filtration on $A[-(n+1)]$ is the canonical filtration (i.e. the generator is in filtration $n+1$). After tensoring with $\FF_\ell$, we get
\[\RR\Hom(A[n],\FF_\ell)\simeq \RR\Hom(A,\FF_\ell)[-n]\simeq \FF_\ell\iota_n\oplus\FF_\ell\beta\iota_n\]
where both generators are in filtration level $n+1$. It follows that 
\[E_1^{**}(\RR\Hom(A[n],\FF_\ell),\tau_\ell)\cong\FF_\ell\iota_n\oplus\FF_\ell\beta\iota_n\]
with $\iota_n$ in bidegree $(-(n+1),2n+1)$ and $\beta\iota_n$ in bidegree $(-(n+1),2n+2)$. 

On the other, hand $E_1(\tilde{C}^*(K(A,n),\FF_\ell),\tau_\ell)\cong H^*(K(A,n),\FF_\ell)$ and since the diagonal map in the triangle sends $\iota_n$ to $\iota_n$ and $\beta\iota_n$ to $\beta\iota_n$ (see notation of Theorem \ref{theo : cohomology of EM spaces}) we have 
\[E_1^{*,*}(\tilde{C}^*(K(A,n),\FF_\ell),\tau_\ell)\cong 
H^*(K(A,n),\FF_\ell).\]
with $\iota_n$ in bidegree $(-(n+1),2n+1)$ and $\beta\iota_n$ in bidegree $(-(n+1),2n+2)$. The bidegree of the other generators will not concern us here. 

Finally let us denote by $B$ the filtered complex $\FF_\ell\iota_n\oplus\FF_\ell\beta\iota_n$ with generators in cohomological degree $n$ and $n+1$ respectively and both in filtration level $(n+1)$. The filtered algebra $F(A,n)\otimes\FF_\ell$ is given by
\[\bigoplus_{k\geq 1}\mathrm{Sym}^k_{\FF_\ell}(B).\]
It follows that, if $i$ is not a multiple of $n+1$
\[E_1^{-i,*}(F(A,n)\otimes\FF_\ell))\cong 0\]
and if $i=-k(n+1)$ we have
\[E_1^{-k(n+1),*+k(n+1)}(F(A,n)\otimes\FF_\ell))\cong H^{*}(\mathrm{Sym}^k_{\FF_\ell}(B))\]
but we know that $i\leq r+2\ell-3$ and that $n\geq r\geq 2$, thus we have 
\[k\leq \frac{r+2\ell-3}{n+1}=\frac{r}{n+1}+\frac{2\ell-3}{n+1}< 1+\frac{2\ell-3}{2}<\ell\]
Therefore, in this range, taking $\Sigma_k$ orbits is exact and we have
\[E_1^{-k(n+1),*+k(n+1)}(F(A,n)\otimes\FF_\ell))\cong \mathrm{Sym}^k_{\FF_\ell}(H^*(B).\]
It follows that, in the range $i\leq r+s$, we have an isomorphism
\[E_1^{*,*}(F(A,n)\otimes\FF_\ell))\cong\FF_\ell[\iota_n,\beta\iota_n]\]
with $\iota_n$ in bidegree $(-(n+1),2n+1)$ and $\beta\iota_n$ in bidegree $(-(n+1),2n+2)$. Moreover the map induced by the vertical map in the above triangle sends $\iota_n$ to $\iota_n$ and $\beta\iota_n$ to $\beta\iota_n$ since this is the case for both the diagonal map and the horizontal map. It follows that this vertical map does indeed induce an isomorphism on $E_1$ in the desired range.
\end{proof}

Let $\cat{EM}_r\st$ denote the full subcategory of $\cat{S}_r\st$ on objects that split as products of Eilenberg--MacLane spaces. Denote by $\cat{CAlg}_r\st$ the $\infty$-category of $(r-1)$-connected non-unital commutative algebras up to $(r,s)$-tame quasi-isomorphisms. Recall that we are using an $\infty$-categorical terminology here so what we call a commutative algebra is typically modelled by an $E_\infty$-algebra at the level of cochain complexes.

\begin{proposition}\label{coro : Mandell EM}
Let $r\geq 1$. The functor $X\mapsto (\tilde{C}^*(X,\ZZ_\ell),\tau)$ from $\cat{EM}_r\st$ to $\cat{CAlg}_r\spt$ is a fully faithful symmetric monoidal functor.
\end{proposition}

\begin{proof}
The fact that it is symmetric monoidal follows from the K\"unneth formula. Consider two products of Eilenberg--MacLane spaces $X$ and $Y$. We wish to show that the map
\[f:\Map_{\cat{S}_r}(X,Y)\to \Map_{\cat{CAlg}_r\spt}(\tilde{C}^*(Y),\tilde{C}^*(X)) \]
is a weak equivalence. Without loss of generality, we can assume that $Y$ is $K(\ZZ_\ell,n)$ or $K(\ZZ/\ell^k,n)$ with $n<r+s$ or $K(\QQ_\ell,n)$ with $n\geq r+s$. Indeed, both the source and target preserve products in the $Y$ variable. We do the first two cases as the last one is more classical. Let $A$ denote $\ZZ_\ell$ or $\ZZ/\ell^k$. By the previous lemma, the target reduces to
\[\Map_{\cat{FCh}_r}((A^\vee[-n],\tau),T(\tilde{C}^*(X),\tau))\]
where $A^\vee$ is our notation for the derived dual of $A$ and $T$ is the $(r,s+1)$-tamification functor (see Remark \ref{rem: tamification}).
Now, we distinguish two cases. If $A=\ZZ_\ell$, then $(A^\vee[-n],\tau)\simeq A(n)[-n]$ and if $A=\ZZ/\ell^k$, then $(A^\vee[-n],\tau)\simeq A(n+1)[-n-1]$ (in either case, the number between round brackets denotes the filtration level). Since $n\leq r+s-1$, we know that $n$ is smaller than the first weight in which the target has become rationalized. It follows that this mapping space reduces to
\[\Map_{\cat{FCh}_r}(A(n)[-n],(\tilde{C}^*(X),\tau))\]
in the first case or
\[\Map_{\cat{FCh}_r}(A(n+1)[-n-1],(\tilde{C}^*(X),\tau))\]
in the second case. These mapping spaces reduce further to 
\[\Map_{\cat{Ch}}(A[-n],\tau_n \tilde{C}^*(X))\]
or
\[\Map_{\cat{Ch}}(A[-n-1],\tau_{n+1}\tilde{C}^*(X)).\]
Finally, we can reduce to 
\[\Map_{\cat{Ch}}(A[-n],\tilde{C}^*(X))\]
or
\[\Map_{\cat{Ch}}(A[-n-1],\tilde{C}^*(X)).\]
using the fact that $\tau_n$ is right adjoint to the inclusion of $(-n)$-connective chain complexes into chain complexes. Now applying the functor $\RR\Hom(-,\ZZ_\ell)$ (which is an equivalence by our finite type condition), this becomes
\[\Map_{\cat{Ch}}(\tilde{C}_*(X),A[n])\]
in both cases. This last mapping space is weakly equivalent to $\Map_{\cat{S}_r}(X,K(A,n))$ and following through the chain of weak equivalences above, one checks easily that the map $f$ is homotopic to the identity map.
\end{proof}

We can finally prove our Mandell-type Theorem in tame homotopy theory. 

\begin{theorem}\label{theo : tame mandell}
Let $r\geq 3$ and $s\leq 2\ell-3$ or $r= 2$ and $s\leq \ell-2$. Let $X$ and $Y$ be two $(r-1)$-connected $\ell$-complete based spaces with finite type cohomology. If there exists an $(r,s+\lfloor\frac{s+1}{r-1}\rfloor)$-tame quasi-isomorphism of $E_\infty$-algebras
\[(\tilde{C}^*(X,\ZZ_\ell),\tau)\simeq (\tilde{C}^*(Y,\ZZ_\ell),\tau)\]
then $X$ and $Y$ have the same $(r,s)$-tame homotopy type.

\end{theorem}

\begin{proof}
First, we observe that the solution of the equation in $x$
\[s=x-\lfloor\frac{x+1}{r}\rfloor\]
is $x=s+\lfloor\frac{s+1}{r-1}\rfloor$. Thus, using Proposition \ref{prop : bar construction}, the $(r,s+\lfloor\frac{s+1}{r-1}\rfloor)$-tame quasi-isomorphism $\tilde{C}^*(X)\simeq \tilde{C}^*(Y)$ induces a $(r-1,s)$-tame quasi-isomorphism of comonoids in $\cat{CAlg}_{r-1}$
\[\mathrm{Bar}(\tilde{C}^*(X))\simeq \mathrm{Bar}(\tilde{C}^*(X)).\]
Since $\tilde{C}^*(X)$ and $\tilde{C}^*(Y)$ are at least simply-connected, it follows that
\[\mathrm{Bar}(\tilde{C}^*(X))\simeq \tilde{C}^*(\Omega X)\;\mathrm{and}\;\mathrm{Bar}(\tilde{C}^*(Y))\simeq \tilde{C}^*(\Omega Y)\]
So we deduce that $\tilde{C}^*(\Omega X)$ and $\tilde{C}^*(\Omega Y)$ are $(r-1,s)$-tamely quasi-isomorphic as comonoids in commutative algebras. Therefore, by Corollary \ref{coro : Mandell EM} and Proposition \ref{prop : product of EM}, we know that $\Omega X$ and $\Omega Y$ are $(r-1,s)$-tame weakly equivalent as loop spaces. It follows that $X$ and $Y$ are $(r,s)$-tame weakly equivalent.
\end{proof}

\begin{remark}
Note that, in the above theorem, the tame parameter on the cochain side is higher than the tame parameter on the space side. We think of this as an ``exchange rate'' that needs to be paid when going from the algebra to the space. This exchange rate depends on the connectivity of the space.
\end{remark}

We isolate the simply-connected case ($r=2$) in the following corollary as it gives the least awkward bound.

\begin{corollary}\label{coro : tame mandell simply-connected}
Let $s\leq \ell-2$. Let $X$ and $Y$ be two simply-connected $\ell$-complete based spaces with finite type cohomology. If there exists an $(2,2s+1)$-tame quasi-isomorphism
\[(\tilde{C}^*(X,\ZZ_\ell),\tau)\simeq (\tilde{C}^*(Y,\ZZ_\ell),\tau)\]
then $X$ and $Y$ have the same $(2,s)$-tame homotopy type.
\end{corollary}

\begin{remark}
 As we will see in Example \ref{Example : projective space}, the statement in the above corollary gives an optimal formality result for the projective space.
\end{remark}

\begin{example}
Consider the space $X=BU(n)$. Then $H^*(X,\ZZ_\ell)\cong \ZZ_\ell[c_1,\ldots,c_n]$. Let $Y=\prod_{i=1}^n K(\ZZ_\ell,2i)$. Let $A=\mathrm{Sym}(\bigoplus_{i=1}^n\ZZ_\ell[-2i],\tau)$ where $\mathrm{Sym}$ denotes the free non-unital $E_\infty$-algebra monad on filtered complexes. Then proceeding as in the proof of Lemma \ref{lemm : case of EM spaces}, we can construct a $(2,2\ell-3)$-tame quasi-isomorphism from $A$ to $\tilde{C}^*(X,\ZZ_\ell)$ and from $A$ to $\tilde{C}^*(Y,\ZZ_\ell)$. It follows from the previous corollary that $X$ and $Y$ have the same $(2,\ell-2)$-tame homotopy type. Similarly, for $n\geq 3$ and $\ell\geq 3$, it can be shown that the space $BSO(n)$ is $(4,2\ell-3-\lfloor\frac{2\ell-1}{4}\rfloor)$-tame weakly equivalent to a product of Eilenberg--MacLane spaces. This reasoning generalizes to any space with polynomial cohomology.
\end{example}

\subsection{Tame formality}
In this section, we introduce the notion of tame formality for $E_\infty$-algebras.

\begin{definition}
An $(r-1)$-connected filtered non-unital algebra $(A,W)$ over $\ZZ_\ell$ is said to be \textit{$(r,s)$-tame homotopically split} if
there is a string of $(r,s)$-tame quasi-isomorphisms of algebras from
$(A,W)$ to $(Gr^W_*A,W)$, with
\[W_iGr^W_*A^n:=\bigoplus_{j\leq i}Gr_j^WA^n.\]
\end{definition}

\begin{remark}
Let $A$ be an $(r-1)$-connected algebra over $\ZZ_\ell$. Then $(A,\tau)$ is $(r,s)$-tame homotopically split if and only if it is $(r,s)$-tame equivalent to $(H^*(A),\tau)$, where $\tau_pH^n(A)=H^n(A)$ when $p\geq n$ and $0$ otherwise.
\end{remark}

\begin{definition}
 An $(r-1)$-connected algebra $A$ over $\ZZ_\ell$ is called \textit{$(r,s)$-tamely formal} if the filtered algebra $(A,\tau)$ is $(r,s)$-homotopically split.
\end{definition}

We record two consequences of tame formality for spaces  in the following proposition.

\begin{proposition}\label{prop : tame formality}
Let $r\geq 3$ and $s\leq 2\ell-3$ or $r= 2$ and $s\leq \ell-2$. Let $X$ be an $(r-1)$-connected based space with finite type integral cohomology, and assume that $\widetilde C^*(X,\ZZ_\ell)$ is $(r,s+\lfloor\frac{s+1}{r-1}\rfloor)$-tamely formal as an $E_\infty$-algebra. Then:
\begin{enumerate}
 \item The $(r,s)$-tame homotopy type of $X$ is a formal consequence of the cohomology ring $H^*(X,\ZZ_\ell)$.
\item If $X$ has the homotopy type of an $N$-dimensional CW complex, with $N\leq s$,
then the $\ell$-adic homotopy type of $X$ is a formal consequence of the cohomology ring $H^*(X,\ZZ_\ell)$.
\end{enumerate}
\end{proposition}

\begin{proof}
The first statement follows directly from Theorem \ref{theo : tame mandell}.
The second statement follows from the first together with  Lemma \ref{CWtruncate}.
\end{proof}

\subsection{Truncated algebras}

From the tame homotopy type of a space we can produced a truncated algebra and there is a natural notion of truncated formality which we introduce and study below.

Informally, an $n$-truncated algebra is the structure that exists on the truncation in weight $\leq n$ of a filtered algebra. Formally it is given by the following definition.

\begin{definition}
An \textit{$n$-truncated algebra} over a commutative ring $R$ is a diagram of cochain complexes over $R$
\[A_0\to A_1 \to\ldots\to A_n\]
with multiplication maps
$A_i\otimes A_j\to A_{i+j}$
defined for any $i+j\leq n$,
and a unit
$R\to A_0$
satisfying a naturality, associativity and unitality constraints. The naturality states that for all $ 0\leq i<i'\leq n$ and all $0\leq j\leq j'\leq n$, the following diagram commutes
\[\xymatrix{
A_i\otimes A_j\ar[r]\ar[d]&A_{i'}\otimes A_{j'}\ar[d]\\
A_{i+j}\ar[r]&A_{i+j'}
}
\]
where the horizontal maps are induced by the structure maps and the vertical maps are the multiplication maps. Unitality and associativity are straightforward.
\end{definition}

\begin{remark}
We could have made a similar definition for other types of algebras (commutative, $E_\infty$, Lie, etc.). In this section we restrict to associative and non-unital associative algebras. The definition of non-unital associative algebras is as above but removing any mention of the unit.
\end{remark}

For any $n$, there is a functor from filtered algebras to $n$-truncated algebras, sending $(A,W)$ to $W_nA$.
In particular, given an algebra $A$, we obtain a canonical $n$-truncated algebra $\tau_{n}A$ by considering the $n$-truncation of the canonical filtration.
We next characterize the homotopical information captured by
$\tau_{n}A$.

\begin{definition}
A morphism of cochain complexes $f:A\to B$ is called \textit{$n$-quasi-isomorphism} if the induced morphism in homology $H^i(f):H^i(A)\to H^i(B)$ is an isomorphism for all $i\leq n$.
\end{definition}

Note that if $f:A\to B$ is an $n$-quasi-isomorphism, then $\tau_{n}A\to \tau_{n}B$ is a quasi-isomorphism of $n$-truncated algebras.
The converse does not hold, but we have a partial converse under some connectivity assumptions:

\begin{proposition}\label{prop:truncationshtpy}Let $r\geq 2$. Let $A$ be an $(r-1)$-connected non-unital dg algebra.
 The $n$-truncated algebra $\tau_{n} A$ determines the $N$-homotopy type of $A$, for $N=\left\lceil\frac{n(r-1)}{r}\right\rceil$.
\end{proposition}

\begin{proof}
First, we apply the bar construction to $A$
\[\mathrm{Bar}(A)=(T^c(sA),d+d')\]
where the differential is the sum of the differential $d$ of $A$ and the bar differential $d'$. Recall that $T^c$ is the cofree coalgebra comonad given by
\[T^c(X)=\bigoplus_{i\geq 1}X^{\otimes i}\]
The differential makes $\mathrm{Bar}(A)$ into a differential graded coalgebra. We can give it a filtration
\[W_1\mathrm{Bar}(A)\subset W_2\mathrm{Bar}(A)\subset\ldots\subset \mathrm{Bar}(A)\]
where $W_n\mathrm{Bar}(A)$ is the subcomplex spanned by elementary tensors $sx_1\otimes \ldots\otimes sx_k$ such that $\sum_{i=1}^k|x_i|\leq n$.
We observe that $W_n\mathrm{Bar}(A)$ is in fact a subcoalgebra of $\mathrm{Bar}(A)$. Indeed, assume that $u=sx_1\otimes\ldots\otimes sx_k\in W_n\mathrm{Bar}(A)$. Then
\[\Delta(u)=\sum_{p=1}^k(sx_1\otimes\ldots \otimes sx_{p})\otimes (sx_{p+1}\otimes\ldots\otimes sx_k)\]
and we observe that each summand of the right-hand side lives in a tensor product of the form $W_p\mathrm{Bar}(A)\otimes W_{n-p}\mathrm{Bar}(A)$ so in particular in $W_n\mathrm{Bar}(A)\otimes W_n\mathrm{Bar}(A)$. Moreover, it is clear that the coalgebra $W_n\mathrm{Bar}(A)$ depends only on the truncated algebra $\tau_{n}(A)$. Finally, let us study the connectivity of the map
\[W_n\mathrm{Bar}(A)\to \mathrm{Bar}(A)\]
For this we observe that the quotient $Gr^k:=W_k\mathrm{Bar}(A)/W_{k-1}\mathrm{Bar}(A)$ has cohomology given by
\[H^i(Gr^k)=H^k(A^{\otimes (k-i)})\]
In particular the cohomology is zero when $i<\frac{k(r-1)}{r}$.
It follows by induction that the map
\[W_k\mathrm{Bar}(A)\to \mathrm{Bar}(A)\]
induces an isomorphism in cohomology in degree $<\frac{k(r-1)}{r}$.

So now, let us denote $B'=W_n \mathrm{Bar}(A)$ and $B=\mathrm{Bar}(A)$. We have a map of coalgebras
\[B'\to B\]
which is an isomorphism in degree $<N:=\frac{n(r-1)}{r}$. We can now apply cobar construction to this map. Recall that the cobar construction of a coalgebra $C$ is given by
\[\Omega(C)=(T(s^{-1}C),d+d')\]
where $d$ is the differential of $C$ and $d'$ is the cobar differential. The cobar construction has a decreasing filtration
\[F_n\Omega(C):=\oplus_{i\geq n}(s^{-1}C)^{\otimes i}\]
and we observe that, for any coconnective coalgebra, $F_n(C)$ is at least $n$-coconnected. Since the map $B'\to B$ induces an isomorphism in cohomology in degree $<N$, the induced map
\[(s^{-1}B')^{\otimes k}\to (s^{-1}B)^{\otimes k}\]
is an isomorphism in degree $<N+1$. Thus, by induction, the map
\[\Omega(B')/F_n\Omega(B')\to \Omega(B)/F_n\Omega(B)\]
is an isomorphism in this range for all $n$. Taking the limit (which is allowed because the connectivity of $F_n(B)$ and $F_n(B')$ grows to infinity), we get that
\[\Omega(B')\to \Omega(B)\]
is an isomorphism in cohomology in degree $<N+1$. But now the source of this map is a dg-algebra which only depends on $\tau_{n}(A)$ and the target is quasi-isomorphic to $A$. This concludes the proof.
\end{proof}

Note in particular that for $r=1$ the above result is empty but in fact, in the non-simply-connected case there is still some homotopical information to salvage, which we record in the following  proposition.

\begin{proposition}
Let $A$ be a dg-algebra. Then the Massey products of any length whose source is of degree $\leq n$ only depend on $\tau_{n}A$ as an $n$-truncated algebra.
\end{proposition}

\begin{proof}
We consider the Eilenberg-Moore spectral sequence computing the bar construction of $A$. This spectral sequence has $E_1$-page given by
\[E_1^{s,t}=H^s(A^{\otimes t})\]
and converges to the bar construction of $A$. It can be constructed as the spectral sequence of the double complex associated to the usual simplicial diagram computing the Bar construction
\[[k]\mapsto A^{\otimes k}\]
It is customary to write the element $x_1\otimes\ldots\otimes x_t$ in $E_1^{s,t}$ as $[x_1|\ldots|x_t]$. It is a theorem of May (see \cite[Theorem 6]{maycohomology}) that, if a Massey product $\langle x_1,\ldots,x_t\rangle$ is defined, then the corresponding element $[x_1|\ldots|x_t]$ in $E_1^{s,t}$ survives to the $E_{t-1}$-page of the spectral sequence. Moreover, for any $y\in E_1^{*,1}$ which is such that
\[y\in \langle x_1,\ldots,x_t\rangle\]
we have the equality
\[y=d_{t-1}[x_1|\ldots|x_t].\]

Now, observe that the data of $\tau_{n}A$ as an $n$-truncated algebra allows one to form the simplicial object
\[[k]\mapsto \tau_{n}A^{\otimes k}\]
whose associated spectral sequence is denoted $'E$. We observe that $'E_1^{s,t}=E_1^{s,t}$ if $s\leq n$ and is zero otherwise. Consider a defined Massey product $\langle x_1,\ldots ,x_t\rangle$ in $H^*(A)$ such that
\[s=\sum_{i=1}^t |x_i|\leq n.\]
Then, the element $[x_1|\ldots|x_t]\in E_1^{s,t}$  can also be viewed as an element of $'E_1^{s,t}$. Moreover, it survives to the $E_{t-1}$-page of either spectral sequences. Finally, the value of the differential $d_{t-1}$ on that element can be computed in $'E_{t-1}$ since that differential is of bidegree $(1-t,-t)$.
\end{proof}

We now compare the two natural notions of formality arising from the homotopy theory of $n$-truncated algebras and from the $n$-homotopy type given by inverting $n$-quasi-isomorphisms.

\begin{definition}Let $A$ be a dg-algebra over $R$ and $n\geq 0$ an integer.
\begin{enumerate}
 \item We say that $A$ is \textit{$n$-formal} if there is a zig-zag of $N$-quasi-isomorphisms of dg-algebras connecting $A$ to its cohomology.
 \item We say that $A$ is \textit{$\tau_{n}$-formal} if there is a zig-zag of quasi-isomorphisms of $n$-truncated algebras connecting $\tau_{n}A$ to $\tau_{n}H(A)$.
\end{enumerate}
\end{definition}

The following result follows from Proposition \ref{prop:truncationshtpy}.
\begin{proposition}
 Let $A$ be an $(r-1)$-connected $\tau_{n}$-formal dg-algebra, with $r\geq 1$.
Then it is $N$-formal, with $N=\left\lceil\frac{n(r-1)}{r}\right\rceil$.
\end{proposition}

\begin{remark}
To connect this section with the previous one, we observe that $(r,s)$-tame formality of an algebra $A$ implies $\tau_{r+s-1}$-formality. So from the previous corollary, we can deduce that $(r,s)$-tame formality implies $\lceil \frac{(r+s-1)(r-1)}{r}\rceil$-formality.
\end{remark}

Note that in an $n$-formal algebra, Massey products whose target is of degree $\leq n$ are zero whereas in a $\tau_{n}$-formal algebra, Massey products whose source is of total degree $\leq n$ vanish.
In particular, in a $1$-formal algebra there are no Massey products from $H^1$ to $H^2$ regardless of their length.

\begin{example}[c.f. \cite{CiHo3}]
The following is an example of a $0$-connected $\tau_{3}$-formal algebra which is not $1$-formal.
Let $A'$ be the free commutative algebra on generators $x_1,\cdots,x_4$ all of degree 1, with $dx_3=x_1x_2$ and $dx_4=x_1x_3$. Consider the quotient $A=A'/(x_2x_3)$. Its cohomology is
generated by $[x_1]$, $[x_2]$ in degree 1, $[x_1x_4]$ and $[x_2x_4]$ in degree 2 and $[x_1x_3x_4]$ in degree 3. In particular, the ring structure is trivial, as well as all triple Massey products. However, there is a 4-tuple Massey product from $H^1$ to $H^2$, given by
$\langle[x_1],[x_1],[x_1],[x_2]\rangle=[x_1x_4]$.
This shows $A$ is not $1$-formal. To prove that it is $\tau_{3}$-formal one builds a filtered model $M\to A$ as follows: Let
\[M=\Lambda\left(x_1,x_2,x_3,x_4,y_{1},y_{2},z_{23},z_{11},z_{12}, z_{21}, z_{22}, w\right)\]
with the following weights and degrees
\begin{center}
\begin{tabular}{l|l|l|l|l|l|l|l|l|l|l|l|l|l|l}
&$x_1$&$x_2$&$x_3$&$x_4$&$z_{23}$&$z_{11}$&$z_{12}$&$z_{21}$&$z_{22}$&$y_{1}$&$y_{2}$&$w$\\
\hline
\text{degree}&1&1&1&1&1&2&2&2&2&2&2&3\\
\hline
\text{weight}&1&1&2&3&3&3&3&3&3&2&2&3\\
\end{tabular}
\end{center}

The differential is given by
\[d(x_3)=x_1x_2,\, d(z_{23})=x_2x_3\text{ and }d(z_{ij})=x_i y_{j}\text{ for }i,j\in \{1,2\}.\]
Define a map $M\to A$ by letting $x_i\mapsto x_i$, $y_i\mapsto x_ix_4$ and $w\mapsto x_1x_3x_4$.
The image of $z_{ij}$ is determined by compatibility with differentials, up to adding cocycles. This map is a quasi-isomorphism up to weight $3$. It is now an easy exercise to
define a quasi-isomorphism from $M$ to its cohomology, up to weight $3$.
\end{example}

\subsection{On the fundamental group}\label{subsec: fundgp}

Given a based connected $\ell$-complete based space $X$, then we can consider the non-unital $E_\infty$-algebra $(\tilde{C}^*(X,\ZZ_\ell),\tau)$ up to $(1,s)$-tame quasi-isomorphism. In this case, this algebra is not enough data to recover the $\ell$-completion of the fundamental group. Nevertheless, it does contain some information about that group as we now explain.

We first revisit a theorem of Stallings on the relationship between the fundamental group of a topological space and the bar construction of its singular cochains.
This relationship is expressed in terms of polynomial maps on a group
as introduced by Passi \cite{Passi}.

Let $G$ be a discrete group, then we may form the group algebra $\ZZ_\ell[G]$ and consider its quotient by powers of the augmentation ideal $\ZZ_\ell[G]/I^k$. The module $\mathrm{Poly}_{\leq k}(G,\ZZ_\ell)$ of polynomial maps of degree $k$ is defined as
\[\mathrm{Poly}_{\leq k}(G,\ZZ_\ell):=\mathrm{Hom}_{\ZZ_\ell}(\ZZ_\ell[G]/I^{k+1},\ZZ_\ell).\]
The ring of polynomial maps is defined as
\[\mathrm{Poly}(G,\ZZ_\ell):=\mathrm{colim}_k\mathrm{Poly}_{\leq k}(G,\ZZ_\ell).\]
Observe that this is a subring of the ring of functions $\Map(G,\ZZ_\ell)=\mathrm{Hom}_{\ZZ_\ell}(\ZZ_\ell[G],\ZZ_\ell)$.

\begin{remark}
For many groups of interest, polynomial maps of degree $\leq k$ are maps $G\to \ZZ_\ell$ that factor through $G/\Gamma^{k+1} G$, with $\Gamma^*G$ the lower central series filtration. Such groups are said to satisfy \textit{the dimension subgroup property}. Examples include free groups, pure braid groups, and more generally, any group $G$ such that $\Gamma^{k}G/\Gamma^{k+1}(G)$ is torsion free for all $k$ (see \cite[Corollary 4.2]{quillenassociated}).
\end{remark}

Let $(X,x)$ be a pointed connected space and $A=\tilde{C}^*(X,\ZZ_\ell)$ viewed as a non-unital dg-algebra using the base point $x$. There is a canonical map
\[\mathrm{Bar}(A)\to \tilde{C}^*(\Omega_xX,\ZZ_\ell).\]
This map is not a quasi-isomorphism in general. It is one if $X$ is simply-connected but it fails to be one even for $X=S^1$. The induced map
\[H^0(\mathrm{Bar}(A))\to H^0(\Omega_xX,\ZZ_\ell)\cong\mathrm{Map}(\pi_1(X,x),\ZZ_\ell)\]
is quite well-understood, thanks to the following theorem of Stallings.

\begin{theorem}[Stallings]\label{theo : Stallings}
Let $(X,x)$ be a pointed connected finite type space and $A=\tilde{C}^*(X,\ZZ_\ell)$. There is an isomorphism of commutative rings
\[H^0(\mathrm{Bar}(A))\cong\mathrm{Poly}(\pi_1(X,x),\ZZ_\ell)\]
Moreover, the filtration on $H^0(\mathrm{Bar}(A))$ induced by the canonical filtration on $A$ coincides with the degree filtration on polynomial maps.
\end{theorem}

\begin{proof}
Note that, since $\mathrm{Bar}(A)$ is concentrated in non-negative degrees, given a filtration $W$ on $\mathrm{Bar}(A)$, the induced filtration on $H^0(\mathrm{Bar}(A))$ satisfies $W=\Dec W$.
Therefore, thanks to Lemma \ref{lemm : two filtrations}, we can use the skeletal filtration instead of the canonical filtration. We denote by $\mathrm{Bar}_{\leq k}$ the $k$-th term of the skeletal filtration. Stallings in \cite[Theorem 5.4]{stallingsquotients} does not quite phrase the result in that manner. He considers instead the dual case of the cobar construction $\mathrm{Cobar}(\tilde{C}_*(X,\ZZ_\ell))$. Observe that, under our finite type assumption, the coalgebra $\tilde{C}_*(X,\ZZ_\ell)$ is quasi-isomorphic to $A^{\vee}$. Let us denote by $\mathrm{Cobar}_{\leq k}(A^{\vee})$ the $k$-totalization of the cosimplicial diagram given by the cobar construction. Then Stallings shows that the map
\[\ZZ_\ell[\pi_1(X,x)]\to H_0(\mathrm{Cobar}_{\leq k}(A^{\vee}))\]
coincides with the map
\[\ZZ_\ell[\pi_1(X,x)]\to \ZZ_\ell[\pi_1(X,x)]/I^k\]
Taking $\ZZ_\ell$-linear duals, we find that the canonical map
\[H^0(\mathrm{Bar}_{\leq k}(A))\to\mathrm{Map}(\pi_1(X,x),\ZZ_\ell)\]
coincides with the inclusion
\[\mathrm{Poly}_{\leq k}(\pi_1(X,x),\ZZ_\ell)\to \mathrm{Map}(\pi_1(X,x),\ZZ_\ell)\]
(note that taking duals turns limits over $\Delta_{\leq k}$ to colimits since we are in a stable setting). Since filtered homotopy colimits in cochain complexes are exact, we find an isomorphism
\[H^0(\mathrm{Bar}(A))\cong \mathrm{Poly}(\pi_1(X,x),\ZZ_\ell).\qedhere\]
\end{proof}

\begin{definition}
We say that a map $f:(A,W)\to (B,W)$ of filtered commutative algebras over $\ZZ_\ell$ is an \textit{$s$-tame isomorphism} if the induced map $\QQ_\ell\otimes W_iA\to \QQ_\ell\otimes W_i B$ is an isomorphism for any value of $i$ and if the map $W_iA\to W_iB$ is an isomorphism for $i< s$.
\end{definition}

\begin{proposition}
Let $X$ be a based connected $\ell$-complete space, then from the $(1,s)$-tame homotopy type of the $E_\infty$-algebra $\tilde{C}^*(X,\ZZ_\ell)$, we can reconstruct the filtered commutative algebra $\mathrm{Poly}(G,\ZZ_\ell)$ up to $s+1$-tame isomorphism.
\end{proposition}

\begin{proof}
First observe that if $f:A\simeq B$ is a $(1,s)$-tame quasi-isomorphism of non-unital algebras, then we obtain a $(1,s)$-tame quasi-isomorphism of filtered algebras
\[\mathrm{Bar}(A)\simeq \mathrm{Bar}(B)\]
with respect to the filtration $W$ of Lemma \ref{lemm : two filtrations}.
The result is then a direct consequence of this observation and the theorem of Stallings recalled above (Theorem \ref{theo : Stallings}).
\end{proof}

In particular, if $\tilde{C}^*(X,\ZZ_\ell)$ is $(1,s)$-tamely formal as a non-unital $E_\infty$-algebra, then we have a $(1,s)$-tame filtered  quasi-isomorphism
\[(\mathrm{Bar}(A),W)\simeq (\mathrm{Bar}(H^*(A)),W)\]
and taking $0$-th cohomology, we obtain the following proposition.

\begin{proposition}\label{prop : pi1}
Assume that $X$ is a connected based space such that $\tilde{C}^*(X,\ZZ_\ell)$ is $(1,s)$-tamely formal as an $E_\infty$-algebra. Then there is an $(s+1)$-tame isomorphism of commutative rings
\[\mathrm{Poly}(\pi_1(X),\ZZ_\ell)\cong H^0(\mathrm{Bar}(H^*(X,\ZZ_\ell))).\]
In particular, truncating at $s$, we get an isomorphism of commutative rings
\[\mathrm{Poly}_{\leq s}(\pi_1(X),\ZZ_\ell)\cong H^0(\mathrm{Bar}_{\leq s}(H^*(X,\ZZ_\ell))).\]
\end{proposition}

In subsection \ref{subsection : fundamental group}, we will see some examples where this proposition applies.

\section{Tame models of algebraic varieties}\label{tamevars}

Throughout this section we fix $K$ a $p$-adic field with residue field $\FF_q$, so that $q=p^n$ is a power of $p$.
Also, fix $\ell\neq p$ a prime number. We will denote by $h$ the order of $q$  in $\FF_\ell^\times$.
We assume that a choice of an embedding of $K$ in $\CC$ has been made and will denote $X:=\Xx\otimes_{K}\CC$ the base-change of an algebraic variety $\Xx$ defined over $K$. We also use the letter $X$ for the underlying topological space of the complex algebraic variety $X$.

The étale cochain algebra $C^*_{et}(\Xx;\ZZ_\ell)$ of any variety $\Xx$ defined over $K$ is endowed with a Frobenius endomorphism $\varphi$ arising from $\mathrm{Gal}(\overline{K}/K)$. Moreover, there is a natural quasi-isomorphism of $E_\infty$-algebras
\[C^*_{et}(\Xx;\ZZ_\ell)\simeq C^*_{sing}(X,\ZZ_\ell).\]
Such quasi-isomorphism and the splitting provided by the generalized eigenspaces of the Frobenius action in étale cohomology, was used in \cite{CiHo2} and \cite{DCH} to prove formality results up to a certain degree, of the associative algebra of cochains of algebraic varieties satisfying certain purity conditions.
The tame homotopy theory approach allows us to enhance these results to the $\ell$-adic homotopy type of such varieties.

\subsection{Weil complexes in tame homotopy theory}

Let us denote by $\Ww_i$ the  set of Weil numbers of pure weight $i$, as in Definition \ref{defweilnum}. Recall this notion depends on the fixed number $q$, which we omit from notation.

\begin{proposition}\label{prop : Hom is zero}
Let $V$ and $V'$ be two complexes of $\ZZ_\ell[\varphi]$-modules. Assume that:
\begin{enumerate}
\item $H^*(V\otimes R)$ and $H^*(V'\otimes R)$ are degreewise of finite type as $R$-modules for $R=\QQ_\ell$ and $R=\FF_\ell$.
\item The endomorphism $\varphi$ of $H^*(V\otimes\QQ_\ell)$ and $H^*(V'\otimes \QQ_\ell)$ has eigenvalues in $\Ww_i$ and $\Ww_j$ respectively.
\item The endomorphism $\varphi$ of $H^*(V\otimes\FF_\ell)$ and $H^*(V'\otimes\FF_\ell)$ has eigenvalues in $\{(\pm\sqrt{q})^i\}$ and $\{(\pm\sqrt{q})^j\}$ respectively.
\end{enumerate}
Then, if $j-i$ is not divisible by $h$, or if $V\otimes\FF_\ell\simeq 0$, we have
\[\RR\Hom_{\ZZ_\ell[\varphi]}(V,V')\simeq 0.\]
\end{proposition}

\begin{proof}
First of all, for any $V$ and $V'$, we have a canonical map
\[\RR\Hom_{\ZZ_\ell[\varphi]}(V,V')\otimes \QQ_\ell\to \RR\Hom_{\QQ_\ell[\varphi]}(V\otimes \QQ_\ell,V'\otimes\QQ_\ell).\]
This map is a quasi-isomorphism if $V=\ZZ_\ell[\varphi]$. Moreover, both the source and target send colimits in $V$ to limits and commute with suspension. It follows that this natural map is a quasi-isomorphism for any $V$.

There is a similar quasi-isomorphism with $\QQ_\ell$ replaced by $\FF_\ell$. We are thus reduced to proving the following two facts:
\begin{enumerate}
\item If $V$ and $V'$ are chain complexes of $\QQ_\ell[\varphi]$-modules such that $H^*(V)$ and $H^*(V')$ have their eigenvalues  in $\Ww_i$ and $\Ww_j$ respectively, then $\RR\Hom_{\QQ_\ell[\varphi]}\Hom(V,V')=0$.
\item If $V$ and $V'$ are chain complexes of $\FF_\ell[\varphi]$-modules such that $H^*(V)$ and $H^*(V')$ have their eigenvalues in $\{(\pm\sqrt{q})^i\}$ and $\{(\pm\sqrt{q})^j\}$ respectively, and $i-j$ is not divisible by $h$, then $ \RR\Hom_{\FF_\ell[\varphi]}\Hom(V,V')=0$
\end{enumerate}
In either case, we are working in the derived category of $\mathbb{K}[\varphi]$ with $\mathbb{K}$ a field. This ring is of homological dimension $1$ so we may assume that $V$ and $V'$ have homology concentrated in a single degree which can even be assumed to be zero up to shifting.

In either case $V$ and $V'$ split as a direct sum of factors of the form $\KK[\varphi]/(f)$ in which the roots of $f$ are Weil numbers of weight $i$  and $j$ respectively in case (1) or $\{(\pm\sqrt{q})^i\}$ and  in $\{(\pm\sqrt{q})^i\}$  respectively in case (2).

Weil numbers of different weights are necessarily distinct so part (1) follows from
the fact that for coprime elements $f$ and $g$ in a principal ideal domain $R$, we always have
$\RR\Hom_R(R/f,R/g)\simeq 0$. Likewise, the numbers $(\pm\sqrt{q})$ have order $2h$ in $\overline{\FF_\ell}^{\times}$, it follows that the equation
\[(\pm\sqrt{q})^i=(\pm\sqrt{q})^j\]
can have a solution only if $j-i$ is a multiple of $h$. So arguing similarly as before we get part (2).
\end{proof}

We shall also need the following variant of the above Proposition:

\begin{proposition}\label{prop : Hom is zero variant}
Let $V$ and $V'$ be two complexes of $\ZZ_\ell[\varphi]$-modules. Assume that
\begin{enumerate}
\item $H^*(V\otimes R)$ and $H^*(V'\otimes R)$ are degreewise of finite type as $R$-modules for $R=\QQ_\ell$ and $R=\FF_\ell$.
\item The endomorphism $\varphi$ of $H^*(V\otimes\QQ_\ell)$ and $H^*(V'\otimes \QQ_\ell)$ has eigenvalues in $\Ww_{2i}$ and $\Ww_{2j}$ respectively.
\item The endomorphism $\varphi$ of $H^*(V\otimes\FF_\ell)$ and $H^*(V'\otimes\FF_\ell)$ has eigenvalues in $q^i$ and $q^j$ respectively.
\end{enumerate}
Then, if $j-i$ is not divisible by $h$, or if $V\otimes\FF_\ell\simeq 0$, we have
\[\RR\Hom_{\ZZ_\ell[\varphi]}(V,V')\simeq 0\]
\end{proposition}

\begin{proof}
The proof is exactly the same as the proof of \ref{prop : Hom is zero}. The key fact being that the equation
\[q^i=q^j\]
can have a solution only if $j-i$ is a multiple of $h$.
\end{proof}

\begin{definition}
An \textit{$(r,s)$-Weil complex} is a $(r-1)$-connected filtered complex $(C,W)$ of $\ZZ_\ell[\varphi]$-modules
such that:
\begin{enumerate}
\item For each $i<r+s$, $W_iC$ has cohomology which is finitely generated as a $\ZZ_\ell$-module.
\item For each $i$, $W_iC\otimes\QQ_\ell$ has cohomology which is finitely generated as a $\QQ_\ell$-module.
\item The eigenvalues of $\varphi$ acting on $H^*(Gr_i^WC\otimes\FF_\ell)$ are in
$\{(\pm\sqrt{q})^i\}$ for $i< r+s$.
\item The eigenvalues of $\varphi$ acting on $H^*(Gr_i^WC\otimes \QQ_\ell)$ are in $\Ww_i$ for all $i$.
\end{enumerate}
\end{definition}

There is an obvious forgetful functor from $(r,s)$-Weil complexes to $(r-1)$-connected filtered complexes and we define the category $\cat{WCh}\st_r$ to be the localization of the $1$-category of Weil complexes at the $(r,s)$-tame quasi-isomorphisms (see Definition \ref{defi : s-tame weak equivalence}).

We consider the functor
\[E_0:\cat{FCh}_r\st\to\cat{FCh}_r\]
sending a $(r-1)$-connected filtered complex to its associated graded complex viewed as a filtered complex. Precisely
\[W_iE_0(C):=\bigoplus_{k\leq i}Gr_k^WC\]
It is easy to verify that $E_0$ induces a non-unital symmetric monoidal endofunctor in $\cat{WCh}_r\st$. If $r=0$, then $E_0$ is a symmetric monoidal endofunctor.

\begin{theorem}\label{e0formality}
If $s\leq h$, the functor $E_0:\cat{WCh}_r\st\to\cat{WCh}_r\st$ is equivalent to the identity functor $\id$, as non-unital symmetric monoidal $\infty$-functors. If $r=0$, then the functor $E_0:\cat{WCh}_r\st\to\cat{WCh}_r\st$ is equivalent to the identity functor $\id$, as symmetric monoidal $\infty$-functors
\end{theorem}

\begin{proof}
Let us write $\cat{GCh}_r\st$ the symmetric monoidal category of graded complexes of $\ZZ_\ell[\varphi]$-modules that are $(r-1)$-connected (i.e. acyclic in weight $<r$. We view them up to $(r,s)$-tame quasi-isomorphism. By that we mean a map which is a quasi-isomorphism in weight $<r+s$ and a rational quasi-isomorphism in all weights. This definition is designed so that $Gr$ induces a symmetric monoidal functor
\[Gr:\cat{FCh}_r\st\to\cat{GCh}_r\st\]
We claim that this functor is fully faithful. Assuming this for the moment, we see that the statement of the theorem is equivalent to proving an equivalence of symmetric monoidal functors $Gr\circ E_0\simeq Gr\circ \id$ which is tautological.

So let us prove the claim. Let $(C,W)$ and $(D,W)$ be two objects of $\cat{WCh}_r\st$. Up to applying the tamification functor of Remark \ref{rem: tamification}, we can assume that $W_iC$ and $W_iD$ have rational cohomology when $i\geq r+s$. We wish to prove that the map
\[\Map_{\cat{WCh}_r}(C,D)\to\prod\Map_{\ZZ_\ell[\varphi]}(Gr_i(C),Gr_i(D))\]
is an equivalence.

We view both the source and target as functors of $C$ and $D$. Clearly, both the source and target preserve colimits in the $C$ variable. We have an equivalence
\[C\simeq \mathrm{colim}_i W_iC\]
where $W_iC$ is viewed as a filtered object given as follows:
\[W_0C\to\ldots\to W_iC\to W_iC\xrightarrow{\id}W_iC\xrightarrow{\id}\ldots\]
It follows that we can reduce the proof to the case where $C$ has a finite filtration. Now, we observe that we have a cofiber sequence
\[W_{i-1}C\to W_iC\to Gr_iC\]
which can be viewed as a cofiber sequence of filtered objects in which $Gr_iC$ is equipped with the filtration
\begin{align*}
W_j (Gr_iC)&=Gr_iC\;\textrm{if}\;i\leq j\\
         &=0\;\textrm{else.}
\end{align*}
From this observation, using an inductive argument, we see that we may reduce to the case in which $C\cong Gr_i(C)$. For this specific $C$, we see that the map that we are considering is simply
\[\Map_{\ZZ_\ell[\varphi]}(Gr_iC,W_iD)\to \Map_{\ZZ_\ell[\varphi]}(Gr_iC,Gr_iD)\]
In order to prove that this map is an equivalence, it is enough to prove that its fiber is zero. This fiber is simply the mapping space $\Map_{\ZZ_\ell[\varphi]}(Gr_iC,W_{i-1}D)$. By a similar inductive argument, we may reduce to proving that
\[\Map_{\ZZ_\ell[\varphi]}(Gr_iC,Gr_{k}D)=0\]
for all $k<i$. We then distinguish two cases.
\begin{enumerate}
\item If $i-k\geq h$ and $k\geq r$, then we have $i\geq r+s$ which means that  $Gr_iC\otimes\FF_\ell\simeq 0$ and the result follows from Proposition \ref{prop : Hom is zero}.
\item If $i-k< h$, then, in particular $i-k$ is not a multiple of $h$, so the result also follows from Proposition \ref{prop : Hom is zero}.
\end{enumerate}
\end{proof}

\begin{remark}
This should be compared to \cite[Theorem 1.6]{BoCiHo}. In fact, the theorem above (in the case $r=s=0$) is claimed without proof in the paragraph following \cite[Theorem 5.2]{BoCiHo}. This was used to prove equivariant formality of algebraic group actions on smooth projective algebraic varieties.
\end{remark}

There is a variant of the previous theorem for Weil complexes having only even weights:

\begin{theorem}\label{e0formality-variant}
If $s\leq 2h$, the functor $E_0$ is naturally weakly equivalent to the functor $\id$ when restricted to Weil complexes that are concentrated in even weights (i.e. the $i$-th associated graded piece is acyclic for $i$ odd).
\end{theorem}

\begin{proof}
Same proof as Theorem \ref{e0formality}, replacing Proposition \ref{prop : Hom is zero} by Proposition \ref{prop : Hom is zero variant}.
\end{proof}

\begin{remark}
This theorem is not an immediate corollary of Theorem \ref{e0formality} as we have replaced the condition $s\leq h$ by $s\leq 2h$.
\end{remark}

A main consequence of the above theorems is that algebras in Weil complexes are homotopically split in the tame sense:

\begin{corollary}\label{coro : formality}
Let $s\leq h$, let $r\geq 0$ and let $(A,W)$ be an $(r-1)$-connected non-unital $E_\infty$-algebra in the $\infty$-category $\cat{WFCh}_r\st$. Then $A$ is weakly equivalent to $E_0(A)$ as an $E_\infty$-algebra in $\cat{WFCh}_r\st$. If $Gr_i^W(A)\simeq 0$ for any odd $i$, then the same conclusion holds if $s\leq 2h$.
\end{corollary}

\begin{remark}
If $r=0$, then this corollary holds for unital or non-unital $E_\infty$-algebras. If $r>0$, then we have to assume that $A$ is non-unital. 
\end{remark}

\subsection{Tame formality of pure varieties}
We prove a tame version of the well-known guiding principle that purity implies formality.

\begin{definition}\label{defi : s-behaved}We say that an algebraic variety  $\Xx$ over $K$ is \textit{$(r,s)$-tamely pure} if:
\begin{enumerate}
\item Its étale cohomology algebra with $\ZZ_\ell$-coefficients is $(r-1)$-connected.
\item The eigenvalues of $\varphi$ acting on $H^i_{et}(\Xx;\FF_\ell)$ are in $\{(\pm\sqrt{q})^i\}$ for $i<r+s$ and the eigenvalues of $\varphi$ acting on $H^i_{et}(\Xx;\QQ_\ell)$ are Weil numbers of weight $i$ for all $i$.
\end{enumerate}
\end{definition}

\begin{remark}\label{remarkblowup}
This property of tame purity  is stable under blow-ups: given $\Zz\subset \Xx$ a pair of tamely pure varieties, then $Bl_\Zz\Xx$ is also tamely pure. This can be deduced easily from the blow-up formula.
\end{remark}

\begin{remark}
If the variety has good reduction (i.e. admits a smooth and proper lift $\tilde{\Xx}$ to $\mathcal{O}_K$, the ring of integers of $K$), the condition about eigenvalues of $\varphi$ on $\QQ_\ell$-\'etale cohomology is automatic by Deligne's work on the Weil conjecture. Indeed, in that situation, the smooth and proper base change theorem for
étale cohomology gives an isomorphism
\[H^*_{et}(\Xx,\QQ_\ell) \cong H^*_{et}(\tilde{\Xx}\otimes_{\Oo_K}\FF_q,\QQ_\ell)\]
where the action of the Frobenius lift on the left hand side coincides with the action of the
Frobenius of $\FF_q$ on the right hand side. Then the result follows from \cite[Théorème 1.6]{DeWeil1}.

If the cohomology of $\Xx$ is concentrated in even degrees, the tame purity condition is closely related to requiring that the motive of $\Xx$ is mixed Tate in cohomological degree $<r+s$. 
\end{remark}

Following the above remark, we define:

\begin{definition}
An algebraic variety $\Xx$ over $K$ is said to be of \textit{Tate type} if:
\begin{enumerate}
\item The eigenvalues of $\varphi$ acting on $H^i_{et}(\Xx;\QQ_\ell)$ are Weil numbers of weight $i$.
\item The eigenvalues of $\varphi$ acting on $H^i_{et}(\Xx;\ZZ_\ell)$ are in $\{(\pm\sqrt{q})^i\}$.
\end{enumerate}
\end{definition}

A variety of Tate type is $(r,s)$-tamely pure for any $s$ and $r$ any number such that the variety is at least $(r-1)$-connected. Examples of varieties of Tate type include $\PP^n_K$, $\overline{\Mm}_{0,n}$, smooth complete toric varieties. If $Z\subset X$ is an inclusion of varieties of Tate type, then the blow-up $\mathrm{Bl}_ZX$ is also of Tate type by the blow-up formula. Note as well that purity implies the following proposition.

\begin{proposition}\label{prop : no torsion}
Assume that $h>1$. Assume that the eigenvalues of $\varphi$ acting on $H^i_{et}(\Xx;\FF_\ell)$ are in $\{(\pm\sqrt{q})^i\}$ for $i$. Then the cohomology groups $H^i_{et}(\Xx;\ZZ_\ell)$ are torsion free for all $i$.
\end{proposition}

\begin{proof}
Let $i$ be the first degree in which there is torsion in cohomology. Then, by the universal coefficient theorem, there is also torsion in degree $i$ homology. Applying again the universal coefficient theorem shows that there is a non-zero cohomological class in $H^i(\Xx;\FF_\ell)$ and $H^{i+1}(\Xx;\FF_\ell)$ that are connected by a Bockstein morphism. Since Steenrod operations are compatible with Frobenius automorphisms this is a contradiction.
\end{proof}

\begin{theorem}\label{theo : tameprojective}
 Let $\Xx$ be an $(r,s)$-tamely pure variety and $x$ be a $K$-point of $\Xx$. Then $\tilde{C}^*(X,\ZZ_\ell)$ is $(r,s')$-tamely formal as a non-unital $E_\infty$-algebra, where $s'=\min(s,h)$. If the cohomology of $\Xx$ is concentrated in even degrees, then $s'=\min(s,2h)$.
\end{theorem}

\begin{proof}
 The $s$-tamely pure condition ensures that, taking the canonical filtration on $\tilde{C}^*_{et}(\Xx;\ZZ_\ell)$ we obtain an algebra in $(r,s)$-Weil complexes. The result now follows from Corollary \ref{coro : formality}.
\end{proof}

\begin{remark}
Note that the existence of the $K$-point is used to construct a model for the reduced cochains equipped with a Frobenius. The result is about formality of the non-unital $E_\infty$-algebra $\tilde{C}^*(X,\ZZ_\ell)$. But this formality can be promoted to a formality result for $C^*(X,\ZZ_\ell)$ by adding a unit. Explicitly, we can deduce that $C^*(X,\ZZ_\ell),\tau)$ is $(0,s'+r)$ tamely formal as a unital $E_\infty$-algebra. Alternatively, without the existence of a $K$-point we may apply Corollary \ref{coro : formality} in the case $r=0$ and deduce that $C^*(X,\ZZ_\ell),\tau)$ is $(0,s')$-tamely formal. Our subsequent results will always have this assumption about the existence of a $K$-point. This could be dropped in a similar way at the expense of getting a worse tame bound.
\end{remark}

\begin{remark}
Taking $s=0$, the above theorem shows formality of the algebra $C^*(X,\QQ_\ell)$ under the assumption that the eigenvalues of the Frobenius in degree $i$ are Weil numbers of weight $i$. We refer the reader to \cite[Theorem 4.16]{ColineKaledin} for a proof of this fact using different techniques.
\end{remark}

\begin{corollary}\label{coro: pureformaltate}
 Let $(\Xx,x)$ be simply-connected and of Tate type with $x$ a $K$-point. Then the $(2,h)$-tame homotopy type of $X$ is a formal consequence of $H^*(X,\ZZ_\ell)$. Moreover, if the cohomology of $\Xx$ is concentrated in even degrees the same conclusion is true for the $(2,2h)$-tame homotopy type.
\end{corollary}

\begin{example}\label{Example : projective space}
The projective space $\PP^n_K$ is $(2,s)$-tamely pure for any $s\geq 0$ and for any $p$-adic field with $p\neq \ell$. So we can pick $K$ such that the residue field is $\mathbb{F}_p$ where $p$ is a prime number of order $\ell-1$ in $\FF_\ell^{\times}$. Since the cohomology of $\CC\PP^n$ is concentrated in even degrees, the previous corollary implies that the $(2,\ell-2)$-tame homotopy type of $\mathbb{CP}^n$ is determined by the cohomology ring $H^*(\CC\PP^n,\ZZ_\ell)$. If $2n\leq \ell$, the $\ell$-adic homotopy type of $\mathbb{CP}^n$ is determined by the cohomology ring $H^*(\CC\PP^n,\ZZ_\ell)$. This last sentence explicitly means the following: if $X$ is another simply-connected $2n$-dimensional CW-complex whose cohomology ring is isomorphic to $H^*(\CC\PP^n,\ZZ_\ell)$ and whose cochain algebra is $(2,2\ell-2)$-tamely formal, then the $\ell$-completion of $X$ is weakly equivalent to the $\ell$-completion of $\CC\PP^n$.
 \end{example}

\begin{remark}
As the above example illustrates, it is desirable to find a $p$-adic field over which our variety is defined and with $h$ as big as possible. In typical situations, our variety is defined over a number field $K$, that we can assume to be Galois over $\QQ$ up to enlarging it. Then we are looking for a completion $K_{\mathfrak{p}}$ satisfying the following two requirements.
\begin{enumerate}
\item The residue field of $K_{\mathfrak{p}}$ is $\FF_p$.
\item The order of $p$ in $\FF_\ell^{\times}$ is $\ell-1$.
\end{enumerate}
We are not sure if this can be achieved in general but this can be done for all but finitely many values of $\ell$. Indeed, except for finitely many primes $\ell$, we can then assume that $K$ does not contain the cyclotomic field $\QQ[\zeta_\ell]$. Then the Galois group of the composite field $K[\zeta_\ell]$ over $\QQ$ is the product $G\times \FF_\ell^{\times}$ where $G=\mathrm{Gal}(K/\QQ)$. The Chebotarev density theorem implies that there are infinitely many prime numbers $p$ that are unramified in $K[\zeta_\ell]$ and whose Frobenius corresponds to the pair $(1,a)$ through the above isomorphism with $a$ of order $\ell-1$ in $\FF_\ell^{\times}$. Then any prime ideal of $K$ lying over such a prime will give a completion satisfying the two conditions above.
\end{remark}

Theorem \ref{theo : tameprojective} is also valid for a more general notion of purity as in \cite{CiHo1}, in which, for a fixed non-zero rational number $\alpha\neq 0$, the  eigenvalues of $\varphi$ acting on $H^i_{et}(\Xx;\FF_\ell)$ are $\pm {q}^{\alpha i\over 2}$ if $i\leq s$ and the eigenvalues of $\varphi$ acting on $H^i_{et}(\Xx;\QQ_\ell)$ are Weil numbers of weight $\alpha i$ for all $i$. We state the result for $\alpha=2$ and leave the general case to the interested reader.

\begin{theorem}\label{theo : tame2pure}
Assume that $\Xx$ is a smooth variety such that the only eigenvalue of the Frobenius action on $H^i_{et}(\Xx,\FF_\ell)$ is $q^i$ for any $i$. Assume also that $\Xx$ has a $K$-point. Then $\tilde{C}^*_{et}(\Xx,\ZZ_\ell)$ is $(r,h)$-tamely formal, where $(r-1)$ is the connectivity of $\tilde{H}^*_{et}(\Xx,\ZZ_\ell)$.
\end{theorem}

Examples of $(r,s)$-tamely pure varieties, for $\alpha=2$, are complements of hyperplane and toric arrangements, and the uncompactified moduli spaces $\Mm_{0,n}$.

\subsection{Complements of admissible arrangements}
Let $\Xx$ be a smooth variety over $K$, and $\Zz=\{\Zz_i\}_{i\in I}$ an admissible arrangement of codimension $c$ in $\Zz$ (Definition \ref{defarrangement}). Denote \[f:\Uu:=\Xx-\Zz\hookrightarrow \Xx\] the inclusion.
The Weil algebra of Proposition \ref{Weilalg}
is also defined over $\ZZ_\ell$, as an $E_\infty$-algebra. 
Indeed, there is a derived direct image functor
$Rf_*$ in the category of sheaves of $E_\infty$-algebras over $\ZZ_\ell$
(see \cite{ChCi}, see also \cite{Dantxif}). Alternatively, we can use purely $\infty$-categorical technology and the functor $Rf_*$ is then a symmetric monoidal $\infty$-functor.

Analogously to the rational case developed in Section \ref{filteredalgebrarational}, we consider the filtered $E_\infty$-algebra
\[(\Aa_{\ZZ_\ell}(\Uu),W):=\Dec\,\Gamma(\Xx_{\text{pro-ét}},(Rf_*\underline{\ZZ}_\ell,\tau^{(2c-1)})).
\]
This filtered algebra is equipped with a Frobenius endomorphism arising from the action of $\mathrm{Gal}(\overline{K}/K)$. If there exists a $K$-point of $\Uu$, then we can use it to split off the units and build a filtered non-unital $E_\infty$-algebra $\tilde{\Aa}_{\ZZ_\ell}(\Uu)$.

\begin{definition}
We say that the arrangement is \textit{$(r,s)$-tamely pure} if the filtered algebra $(\tilde{\Aa}_{\ZZ_\ell}(\Uu),W)$ with its Frobenius endomorphism is an algebra in $(r,s)$-Weil complexes.
\end{definition}

The following proposition gives two situations in which it is easy to check that this definition is satisfied.

\begin{proposition}
The arrangement is $(r,s)$-tamely pure in the following two cases.
\begin{enumerate}
\item The variety $\Xx$ is smooth projective and $s=0$.
\item For each $\varnothing \subset S\subset I$, the variety $\Zz_S:=\cap_{s\in S}\Zz_s$ is of Tate type and $s$ is any nonnegative number.
\end{enumerate}
In both cases, the parameter $r$ is any positive integer such that $W_{r-1}\Aa(\Uu)$ has trivial cohomology.
\end{proposition}

\begin{proof}
The first claim is Part (2) of Proposition \ref{Weilalg}. The second claim is proven analogously to that proposition.
\end{proof}

The following is a tame version of Proposition \ref{Weilalg}.

\begin{proposition}\label{Weilalgtame}
Let $\Zz=\{\Zz_i\}_{i\in I}$ be an $(r,s)$-tamely pure arrangement in $\Xx$. Then the associated graded
of $(\Aa_{\ZZ_\ell}(\Uu),W)$ 
is quasi-isomorphic as an $E_\infty$-algebra to the $E_{2c}$-term of the Leray spectral sequence for $\underline{\ZZ}_\ell$ relative to the inclusion $f:U\hookrightarrow X$. Likewise, if $u\in \Uu$ is a $K$-point, the associated graded of $(\tilde{\Aa}_{\ZZ_\ell}(\Uu),W)$ is quasi-isomorphic as a non-unital $E_\infty$-algebra to the $E_{2c}$-term of the reduced Leray spectral sequence.
\end{proposition}

\begin{proof}
Weber's result \cite{Weber} on the decomposition of the sheaves $R^{i(2c-1)}f_*\underline{\QQ}_\ell$ is also valid over $\ZZ_\ell$, so we have that 
 $R^kf_*\underline{\ZZ}_\ell$ is 0 when $k$ is not divisible by $2c-1$. This implies that $d_i=0$ for all $i<2c$ and so  ${}^{\Ll}E_{2}^{*,*}(f,\ZZ_\ell)\cong {}^{\Ll}E_{2c}^{*,*}(f,\ZZ_\ell)$. The identification 
\[E_0^{-i,j+i}(\Aa_{\ZZ_\ell}(\Uu),W)\simeq {}^{\Ll}E_{2c}^{j+(j-i)(2c-1),(j-i)(2c-1)}(f,\ZZ_\ell)\text{ and }d_0^W=d_{2c}^{\Ll}\]
follows as in the proof of Proposition \ref{Weilalg}.
\end{proof}

The main result of this section is the following tame version of Theorem \ref{maintheorational}. We first introduce the notation $\tilde{E}_r^{**}$ for the spectral sequence obtained by splitting off the unit of a spectral sequence $E_r^{**}$ associated to an augmented filtered algebra.

\begin{theorem}\label{maintame}
Let $\Zz=\{\Zz_i\}_{i\in I}$ be an $(r',s')$-tamely pure arrangement in $\Xx$, with $s'<h$. Let
\[r=\left\lceil\frac{(2c-1)r'}{2c}\right\rceil\quad\text{ and }\quad s=\left\lfloor\frac{(2c-1)s'}{2c}\right\rfloor.\] 
Let $u\in \Uu$ be a $K$-point. Then there is an $(r,s)$-tame quasi-isomorphism
\[\tilde{C}^*(U,\ZZ_\ell)\simeq ({}^{\Ll}\tilde{E}_{2c}^{*,*},d_{2c}).\]
Moreover, if for any $\varnothing \subset S\subset I$, the cohomology of $\bigcap_{s\in S}\Zz_s$ is concentrated in even degrees, the same conclusion holds under the weaker assumption that $s'<2h$. 
\end{theorem}

\begin{proof}
By definition, the filtered algebra $(\tilde{\Aa}_{\ZZ_\ell}(\Uu),W)$ is a non-unital $E_\infty$-algebra in $(r',s')$-Weil complexes. Therefore by Theorem \ref{e0formality}, if $s'<h$, it is $(r',s')$-tame quasi-isomorphic to its associated graded, with the column filtration. The latter is given by the filtered algebra $(\tilde{E}_{2c}^{*,*},d_{2c})$ with filtration $W$ given by
\[W_i\tilde{E}_{2c}^{n}:=\bigoplus_{j\leq i-n} {}^{\Ll}E_{2c}^{n+j-2jc,j(2c-1)}(f,\ZZ_\ell).\]
Indeed, we have 
\begin{align*}
 W_i\tilde{E}_0^n=\bigoplus_{j\leq i} \tilde{E}_0^{-j,n+j}(W)\simeq \bigoplus_{j\leq i} \tilde{E}_1^{n-j,j}(\Dec^{-1} W)=\\
 =\bigoplus_{j\leq i-n}\tilde{E}_1^{-j,j+n}(\Dec^{-1} W)\cong \bigoplus_{j\leq i-n}{}^{\Ll}\tilde{E}_{2c}^{n+j-2jc,j(2c-1)}(f,\ZZ_\ell).
\end{align*}
Since ${}^{\Ll}\tilde{E}_{2c}^{n+j-2jc,j(2c-1)}(f,\ZZ_\ell)=0$ for 
$n+j-2jc<0$, it follows that
\[W_i\tilde{E}_{2c}^{n}=\tilde{E}_{2c}^n\text{ when }i= \left\lfloor {2cn\over 2c-1}\right\rfloor.\]
This ensures that this filtered algebra fits in Proposition \ref{prop : different filtration}, with $\sigma(n)=\lfloor{2c\over 2c-1} n\rfloor$. So we deduce that there is an $(r,s)$-tame quasi-isomorphism
\[\tilde{C}^*(U,\ZZ_\ell)\simeq ({}^{\Ll}\tilde{E}_{2c}^{*,*},d_{2c})\]
with
\[s:=\mathrm{max}\{n\in\mathbb{N},\sigma(n)\leq s'\}\quad\text{ and }\quad r:=\mathrm{min}\{n\in\mathbb{N},\sigma(n)\geq r'\}.\] It is then a straightforward exercise to check that
\[r=\left\lceil\frac{(2c-1)r'}{2c}\right\rceil\quad\text{ and }\quad
s=\left\lfloor\frac{(2c-1)s'}{2c}\right\rfloor.\qedhere\]
\end{proof}

In the next two subsections, we apply this theorem to two particular situations: smooth complex varieties and configuration spaces of smooth projective varieties.

\subsection{Smooth varieties}
Let $X$ be a smooth complex algebraic variety.
Take $X\hookrightarrow \overline{X}$ a smooth compactification in such a way that $D=\overline{X}-X$ is a simple normal crossings divisor.
The codimension-1 arrangement defined by $D$ in $X$ is defined by a finite number of polynomial equations so by standard spreading out arguments there is a $p$-adic field $K$ embedded in $\CC$, with residue field $\FF_q$ over which the arrangement is defined, and which has  good reduction.
By Theorem \ref{maintame}, if this arrangement is $(r',s')$-tamely pure and if there is a $K$-point in $X$, then there is an $(\lceil\frac{r'}{2}\rceil,\lfloor\frac{s'}{2}\rfloor)$-tame quasi-isomorphism
\[\tilde{C}^*(X,\ZZ_\ell)\simeq ({}^{\Ll}\tilde{E}_{2},d_2).\]
This algebra is exactly the same as in Example \ref{smoothrational}, but with coefficients in $\ZZ_\ell$.

In some situations, the simple form of the tame model $({}^{\Ll}E_{2},d_2)$ allows us to define a direct quasi-isomorphism to its cohomology, thus giving tame formality, as we illustrate in the following simple example.

\begin{example}
Let $X=\CC\PP^2$ and $U$ the complement of three projective lines in general position. For the sake of simplicity, we assume that these lines are defined over $\QQ$. A resolution for $D=X-U$ is given by $D^{(1)}=L_1\sqcup L_2\sqcup L_3$ and $D^{(2)}=\{p_{12}, p_{13}, p_{23}\}$, with $p_{ij}=L_i\cap L_j$ and $L_i=\CC\PP^1$.  This is an $(r,s)$-tamely pure arrangement for any $s\geq 0$ and $r=2$. Each intersection has cohomology concentrated in even degrees. Therefore, a $(1,\ell-1)$-tame model for $U$ is given by the filtered non-unital algebra $({}^{\Ll}\tilde{E}_{2},d_2)$ which, according to Example \ref{smoothrational} is given as follows:
Choose generators $\gamma_i\in H^0(L_i)$ and the corresponding fundamental classes $\ell_i\in H^2(L_i)$.
Let also $u\in H^2(\CC\PP^2)$ denote the Kähler class.
Then $({}^{\Ll}\tilde{E}_{2},d_2)$ is given by:
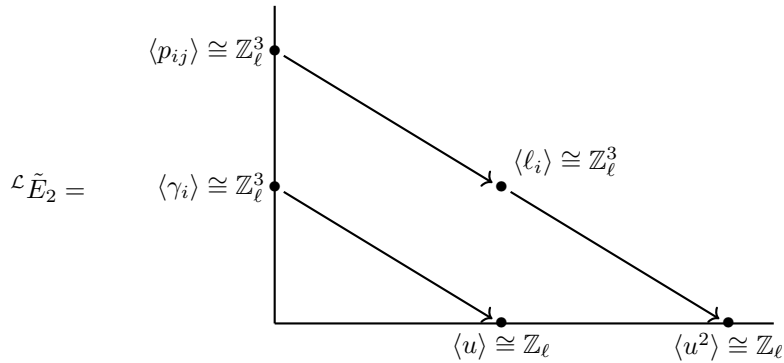
\begin{center}
\begin{tikzpicture}[scale=1.2]
\draw[thick] (0,0) -- (5.5,0);
\draw[thick] (0,0) -- (0,3.5);

\draw (0,1.5) node[left] {$\langle \gamma_i\rangle\cong\ZZ_\ell^3$} node{$\bullet$};
\draw (0,3)   node[left] {$\langle p_{ij}\rangle\cong\ZZ_\ell^3$} node{$\bullet$} ;

\draw (2.5,0) node[below] {$\langle u\rangle\cong \ZZ_\ell$} node{$\bullet$};
\draw (5,0)   node[below] {$\langle u^2\rangle\cong \ZZ_\ell$} node{$\bullet$};

\draw (2.5,1.5) node{$\bullet$} node[above right=1pt] {$\langle \ell_{i}\rangle\cong \ZZ_\ell^3$};

\draw[thick] [->] (0.1, 1.45) -- (2.4, 0.05);
\draw[thick] [->] (0.1, 2.95) -- (2.4, 1.55);
\draw[thick] [->] (2.6, 1.45) -- (4.9, 0.05);

\draw (-2.5, 1.5) node {${}^{\Ll}\tilde{E}_{2}=$};
\end{tikzpicture}
\end{center}
with the products
\[p_{ij}=\gamma_i\cdot \gamma_j\text{ and }\ell_i=u\cdot \gamma_i\]
and the differential $d_2$ determined by $\gamma_i\mapsto u$ and the Leibniz rule.
The cohomology of this algebra is the subalgebra of
$({}^{\Ll}\tilde{E}_{2},d_2)$ given by
\[H\cong \Ker(d_2)\cap {}^{\Ll}\tilde{E}_{2}^{0,*}\cong \ZZ_\ell\langle\gamma_1-\gamma_2,\gamma_1-\gamma_3\rangle \oplus \ZZ_\ell\langle p_{12}-p_{13}+p_{23}\rangle\]
Note the induced product structure in cohomology recovers the cohomology of the torus.
The obvious inclusion from the cohomology to $({}^{\Ll}\tilde{E}_{2},d_2)$ is a quasi-isomorphism so in particular we obtain tame formality.
\end{example}

\subsection{Configuration spaces}

We can apply Theorem \ref{maintame} to configuration spaces. First, from Totaro \cite{Totaro}, we know that, with $\ZZ_\ell$-coefficients, the $E_2$-page of the Leray-Serre spectral sequence of the inclusion $\mathrm{Conf}_n(X)\to X^n$ is the $\ZZ_\ell$-version of the graded algebra described in Example \ref{rationalconfexample}.

\begin{theorem}\label{theo: mainconftame}
Let $\Xx$ be a smooth variety over $K$, of dimension $m$. Assume that $\Xx$ is $(r',s')$-tamely pure in the sense of Definition \ref{defi : s-behaved} and that $s'<h$. Let $u$ be a $K$-point of $\mathrm{Conf}_n(\Xx)$. Let
\[r=\left\lceil\frac{(2m-1)\mathrm{min}(r',2m)}{2m}\right\rceil\quad\text{ and }\quad s=\left\lfloor\frac{(2m-1)s'}{2m}\right\rfloor.\]
Then there is an $(r,s)$-tame quasi-isomorphism of $E_\infty$-algebras between $\tilde{C}^*(\mathrm{Conf}_n(X),\ZZ_\ell)$ and the $\ZZ_\ell$-coefficients version of the commutative dg-algebra described in Example \ref{rationalconfexample}. Moreover, if the cohomology of $X$ is concentrated in even degrees, the same conclusion holds if $s'<2h$.
\end{theorem}

\begin{proof}
Indeed, in this case the lowest weight that appear in $(\Aa_{\ZZ_\ell}(\Uu),W)$ is $\mathrm{min}(r',2m)$ so the arrangement is $(\mathrm{min}(r',2m),s)$ tamely pure and we can apply Theorem \ref{maintame}.
\end{proof}

In the Tate type and simply-connected case, the theorem simplifies as follows.

\begin{corollary}\label{coro: mainconftate}
Let $\Xx$ be a smooth variety over $K$ of Tate type, of dimension $m$ and cohomologically $1$-connected. Let $u$ be a $K$-point of $\mathrm{Conf}_n(\Xx)$. Let
\[r=\left\lceil\frac{2m-1}{2m}\right\rceil\quad\text{ and }\quad s=\left\lfloor\frac{(2m-1)h}{2m}\right\rfloor.\]
Then there is an $(r,s)$-tame quasi-isomorphism of $E_\infty$-algebras between $\tilde{C}^*(\mathrm{Conf}_n(X),\ZZ_\ell)$ and the $\ZZ_\ell$-coefficients version of the commutative dg-algebra described in Example \ref{rationalconfexample}. Moreover, if the cohomology of $X$ is concentrated in even degrees, the same conclusion holds with $s=\left\lfloor\frac{(2m-1)2h}{2m}\right\rfloor$.
\end{corollary}

\begin{corollary}\label{coro: tatetameconfeladic}
Under the assumption of the previous Corollary, if $m\geq 2$. Let $t$ be such that
\begin{enumerate}
\item $t\leq \frac{2mh-h-3m}{4m}$ or,
\item $t\leq \frac{2mh-h-m}{2m}$ if  the cohomology of $X$ is concentrated in even degrees.
\end{enumerate}
Then the $(2,t)$-tame homotopy type of $\mathrm{Conf}_n(X)$ is the formal consequence of the cohomology ring $H^*(X,\ZZ_\ell)$ together with the data of the diagonal class $\Delta\in H^{2m}(X^2,\ZZ_\ell)$.
\end{corollary}

\begin{proof}
This is a direct consequence of the previous Corollary and Corollary \ref{coro : tame mandell simply-connected}.
\end{proof}

In simple situations, the above tame models for configuration spaces lead to formality. This is exhibited in the example below.

\begin{example}Let us consider the space $\mathrm{Conf}_3(\CC\PP^1)$.
 This is homotopy equivalent to $S^3$, but one can  argue similarly for configurations of more points in higher dimensional projective spaces. A $(1,\ell-1)$-tame model of $\mathrm{Conf}_3(\CC\PP^1)$ is given by the algebra

   
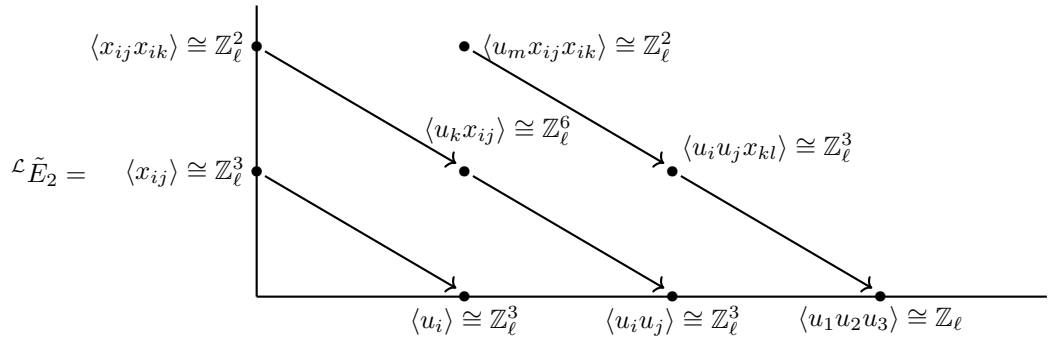
\begin{center}
\begin{tikzpicture}[scale=1.1]
\draw[thick] (0,0) -- (9.5,0);
\draw[thick] (0,0) -- (0,3.5);

\draw (0,1.5) node{$\bullet$} node[ left=1pt] {$\langle x_{ij} \rangle \cong \ZZ_\ell^3$};
\draw (0,3)   node{$\bullet$} node[ left=1pt] {$\langle x_{ij}x_{ik} \rangle \cong \ZZ_\ell^2$};

\draw (2.5,0)   node{$\bullet$} node[below] {$\langle u_i \rangle \cong \ZZ_\ell^3$};
\draw (2.5,1.5) node{$\bullet$};
\draw (2.9,2)  node{$\langle u_k x_{ij} \rangle \cong \ZZ_\ell^6$};
\draw (2.5,3)   node{$\bullet$} node[right=3pt] {$\langle u_m x_{ij}x_{ik} \rangle \cong \ZZ_\ell^2$};

\draw (5,0)   node{$\bullet$} node[below] {$\langle u_i u_j \rangle \cong \ZZ_\ell^3$};
\draw (5,1.5) node{$\bullet$} node[above right] {$\langle u_i u_j x_{kl} \rangle \cong \ZZ_\ell^3$};

\draw (7.5,0) node{$\bullet$} node[below] {$\langle u_1 u_2 u_3 \rangle \cong \ZZ_\ell$};

\draw[thick] [->] (0.1, 1.45) -- (2.4, 0.1);
\draw[thick] [->] (0.1, 2.95) -- (2.4, 1.6);
\draw[thick] [->] (2.6, 1.45) -- (4.9, 0.1);
\draw[thick] [->] (2.6, 2.95) -- (4.9, 1.6);
\draw[thick] [->] (5.1, 1.45) -- (7.4, 0.1);

\draw (-2.5, 1.5) node {${}^{\Ll}\tilde{E}_{2}=$};
\end{tikzpicture}
\end{center}
where $u_i\in H^2(\CC\PP^1_i)$ denotes the fundamental class, for $i=1,2,3$ and $x_{ij}$ are the exterior diagonal classes, for $1\leq i<j\leq 3$.  Moreover, we have the module relations $u_i x_{ij}=u_j x_{ij}$ as well as the Arnol'd relation
 \[x_{12}x_{13}-x_{12}x_{23}+x_{13}x_{23}=0.\]
 The differential $d_2$ is generated by  $x_{ij}\mapsto u_i+u_j$.
 The cohomology is given by ${}^{\Ll}\tilde{E}_{2}^{2,1}\cong \ZZ_\ell$ and the remaining bidegrees are trivial.
 There is an obvious map from the cohomology to the model which is a quasi-isomorphism.
 Therefore we get $(1,\ell-1)$-tame formality.
  \end{example}

\subsection{Fundamental groups}\label{subsection : fundamental group}
Lastly, we study the case of non-simply-connected tamely pure varieties.
\begin{proposition}
Let $\Xx$ be a smooth variety over $K$. Let $x\in \Xx$ be a $K$-point. Assume that $\Xx$ is $(1,s)$-tamely pure in the sense of Definition \ref{defi : s-behaved} with $s\leq h$. Then there is an $(s+1)$-tame isomorphism
\[\mathrm{Poly}(\pi_1(X,x),\ZZ_\ell)\cong H^0(\mathrm{Bar}(H^*(X,\ZZ_\ell)).\]
If $\Xx$ is $(1,s)$-tamely $2$-pure (i.e. its cohomology only has even weights and the $i$-th cohomology group is of weight $2i$), then the same conclusion holds.
\end{proposition}

\begin{proof}
Under this hypothesis, by Theorems \ref{theo : tameprojective} and \ref{theo : tame2pure}, the cochain algebra $\tilde{C}^*(X,\ZZ_\ell)$ is $(1,s)$-tamely formal. The result then follows from Proposition \ref{prop : pi1}.
\end{proof}

\begin{example}
Let $\Uu$ be the complement of a hyperplane arrangement defined over $K$ and $u$ be any $K$-point of $\Uu$, then there is an $h$-tame isomorphism
\[\mathrm{Poly}(\pi_1(U,u),\ZZ_\ell)\cong H^0(\mathrm{Bar}(H^*(U,\ZZ_\ell))\]

In particular, the braid arrangement is defined over $\QQ$ so, in that case, can assume that $K$ has residue field $\FF_p$ with $p$ of order $\ell-1$ in $\FF_\ell^{\times}$. We deduce that there is an $\ell$-tame isomorphism
\[\mathrm{Poly}(P_n,\ZZ_\ell)\cong H^0(\mathrm{Bar}(H^*(\mathrm{Conf}_n(\CC),\ZZ_\ell)).\]
Similarly, applying the previous proposition to $\Mm_{0,n}$, we find an $\ell$-tame isomorphism
\[\mathrm{Poly}(\Gamma_{0,n},\ZZ_\ell)\cong H^0(\mathrm{Bar}(H^*(\Mm_{0,n},\ZZ_\ell))\]
for $\Gamma_{0,n}$ the mapping class group of a genus zero surface with $n$ marked points.
\end{example}

\begin{remark}
It turns out that the cohomology algebra of the complement of the braid arrangement is Koszul. This means that there is an explicit description of the algebra $H^0(\mathrm{Bar}(H^*(U,\ZZ_\ell))$. This is the dual of the Yang-Baxter algebra. The latter algebra is generated by primitive elements $b_{ij},1\leq i\neq j\leq n$ subject to the following relations
\begin{enumerate}
\item $b_{ij}=b_{ji}$,
\item $[b_{ij},b_{jk}]=[b_{jk},b_{ki}]=[b_{ki},b_{ij}]$ when $i,j,k$ are all distinct.
\item $[b_{ij},b_{kl}]=0$ when $\{i,j\}\cap\{k,l\}=\varnothing$.
\end{enumerate} 
\end{remark}
Therefore, our example above can be rephrased as follows. The Hopf algebra $\mathrm{Poly}(P_n,\ZZ_\ell)$ is $(\ell-1)$-tamely isomorphic to the dual of the Yang-Baxter algebra.

\bibliographystyle{amsalpha}
\bibliography{biblio}

\end{document}